\pdfoutput=1
\RequirePackage{snapshot}
\documentclass[final]{siamltex}

\usepackage[utf8]{inputenc}
\usepackage[T1]{fontenc}

\usepackage{amsmath, amsfonts}
\usepackage{bbm}

\usepackage{color}

\usepackage{mathrsfs}

\renewcommand{\hat}{\widehat}
\renewcommand{\tilde}{\widetilde}
\renewcommand{\phi}{\varphi}

\newcommand{\argdot}{\mathbf{\bullet}}

\renewcommand{\i}{i\,}

\newcommand{\N}{\mathbb{N}}

\newcommand{\R}{\mathbb{R}}
\renewcommand{\S}{\mathbb{S}}

\newcommand{\norm}[1]{\left\lVert #1 \right\rVert}
\newcommand{\p}[1]{\left( #1 \right)}

\newcommand{\argmin}{\operatorname{\arg}\min}

\usepackage{aliascnt}
\newaliascnt{conj}{theorem}
\newaliascnt{cor}{theorem}
\newaliascnt{lemma}{theorem}
\newaliascnt{prop}{theorem}
\newaliascnt{definition}{theorem}
\newaliascnt{example}{theorem}
\newaliascnt{notation}{theorem}
\newaliascnt{experiment}{theorem}
\newaliascnt{assumption}{theorem}

\aliascntresetthe{conj}
\aliascntresetthe{cor}
\aliascntresetthe{lemma}
\aliascntresetthe{prop}
\aliascntresetthe{definition}
\aliascntresetthe{example}
\aliascntresetthe{notation}
\aliascntresetthe{experiment}
\aliascntresetthe{assumption}

\renewcommand{\argmin}{\operatorname*{argmin}}

\newtheorem{example}{Example}[section]

\newtheorem{assumption}{Assumption}[section]

\newcommand{\figspace}{\vspace{-0.4cm}}

\usepackage[boxed, ruled, lined, algo2e]{algorithm2e}

\usepackage[usenames,dvipsnames,svgnames]{xcolor}

\usepackage{enumerate}
\usepackage{csquotes}

\usepackage[font=footnotesize,labelfont=bf]{caption}
\usepackage{subfig}

\usepackage[pdftex, 
 colorlinks = true,
  linkcolor= black,
  citecolor = black,
  urlcolor=blue,
  ]{hyperref}
  
\newcommand{\PC}{\mathrm{PC}[0,1]}
\newcommand{\PCk}{\mathrm{PC}^k[0,1]}
\newcommand{\indfunc}{\mathbf{1}}

\newcommand{\de}{\,\mathrm{d}}

\newcommand{\numel}[1]{\# #1}

\newcommand{\noop}{}

\usepackage{tikz}
\usepackage{pgfplots}
\usetikzlibrary{arrows,shapes,snakes,automata,backgrounds,petri}

\usepgfplotslibrary{external} 
\pgfkeys{/pgfplots/MyAxisStyle/.style={
        xtick={0, 128, 256},
         font=\scriptsize,
        width=\thisfigwidth,
        height=\thisfigheight,
        ytick={0, 1},
}}
\pgfkeys{/pgfplots/MyAxisStyleB/.style={
        xtick={0, 128, 256},
        font=\scriptsize,
        ytick = \empty,
        width=\thisfigwidth,
        height=\thisfigheight,
}}

\title{ The $L^1$-Potts functional for robust jump-sparse reconstruction.}

\author{Andreas Weinmann\footnotemark[1], Martin Storath\footnotemark[2], and Laurent Demaret\footnotemark[1]}

\date{{ \today}}

\renewcommand{\thefootnote}{\fnsymbol{footnote}}

\footnotetext[1]{
\noop
Helmholtz Center Munich, and Department of Mathematics, Technische Universit\"at M\"unchen (Germany), 
andreas.weinmann@tum.de,
laurent.demaret@helmholtz-muenchen.de}

\footnotetext[2]{{\noop Biomedical Imaging Group, École Polytechnique Fédérale de Lausanne (Switzerland), martin.storath@epfl.ch}}

\usepackage{xspace}
\usepackage{xparse}

\begin{document}

\maketitle

\renewcommand{\thefootnote}{}

\begin{abstract}
We investigate the non-smooth and non-convex $L^1$-Potts functional in discrete and continuous time. We show $\Gamma$-convergence of discrete $L^1$-Potts functionals
towards their continuous counterpart and obtain a convergence statement for the corresponding minimizers as the discretization gets finer.
For the discrete $L^1$-Potts problem, we introduce an $O(n^2)$ time and $O(n)$ space algorithm 
to compute an exact minimizer.
We apply $L^1$-Potts minimization to the problem of recovering piecewise constant signals from noisy measurements $f.$
It turns out that the $L^1$-Potts functional has a quite interesting blind deconvolution property. 
In fact, we show that mildly blurred jump-sparse signals are reconstructed by minimizing the $L^1$-Potts functional.
Furthermore, for strongly blurred signals and known blurring operator, we derive an iterative reconstruction algorithm.

\end{abstract}

\section{Introduction}

\begin{figure}
\centering
\def\noiseLevel{0.1}
\def\thisfigheight{0.3\textwidth}
\def\thisfigwidth{0.43\textwidth}
\def\thisdatafile{Experiments/introduction/pottsIntroLaplace\noiseLevel.table}
\setlength{\tabcolsep}{1pt}
\begin{tabular}{rrr}
   	\begin{tikzpicture}
      \begin{axis}[
MyAxisStyle,  
ymin=-0.4,
ymax=1.3,
title=Laplacian noise
        ]
        \pgfplotsset{every axis plot post/.append style={line width = 0.5pt}} ;
        \addplot+[only marks, mark size=0.5] table[x index=0,y index=2] {\thisdatafile};   
       
     \end{axis}
    \end{tikzpicture}
    &
     \def\thisdatafile{Experiments/introduction/pottsIntroGauss\noiseLevel.table}
     \begin{tikzpicture}
      \begin{axis}[
      MyAxisStyleB,
      ymin=-0.4,
ymax=1.3,
      title=Gaussian noise
        ]
        \pgfplotsset{every axis plot post/.append style={line width = 0.5pt}} ;
        \addplot+[only marks, mark size=0.5] table[x index=0,y=noisy] {\thisdatafile};      
     \end{axis}
         \end{tikzpicture}
     &
     \def\thisdatafile{Experiments/introduction/pottsIntroSaP0.25.table}
     \begin{tikzpicture}
      \begin{axis}[
      MyAxisStyleB,
      ymin=-0.4,
ymax=1.3,
      title=Impulsive noise
        ]
        \pgfplotsset{every axis plot post/.append style={line width = 0.5pt}} ;
        \addplot+[only marks, mark size=0.5] table[x=x,y index = 2] {\thisdatafile};   
     \end{axis}
    \end{tikzpicture}
    \\
    \def\thisdatafile{Experiments/introduction/pottsIntroLaplace\noiseLevel.table}
	\begin{tikzpicture}
      \begin{axis}[
      MyAxisStyle
        ]
        \pgfplotsset{every axis plot post/.append style={line width = 0.5pt}} ;
        \addplot+[only marks, mark size=0.5] table[x index=0,y = pottsL1] {\thisdatafile};   
        \addplot+[no marks, dashed] table[x index=0,y index=1] {\thisdatafile};   
     \end{axis}
    \end{tikzpicture}
    &
     \def\thisdatafile{Experiments/introduction/pottsIntroGauss\noiseLevel.table}
     \begin{tikzpicture}
      \begin{axis}[
      MyAxisStyleB
        ]
        \pgfplotsset{every axis plot post/.append style={line width = 0.5pt}} ;
        \addplot+[only marks, mark size=0.5] table[x index=0,y = pottsL1] {\thisdatafile};    
        \addplot+[no marks, dashed] table[x index=0,y=original] {\thisdatafile};   
     \end{axis}
         \end{tikzpicture}
     &
     \def\thisdatafile{Experiments/introduction/pottsIntroSaP0.25.table}
     \begin{tikzpicture}
      \begin{axis}[
      MyAxisStyleB
        ]
        \pgfplotsset{every axis plot post/.append style={line width = 0.5pt}} ;
        \addplot+[only marks, mark size=0.5] table[x=x,y = pottsL1] {\thisdatafile};   
        \addplot+[no marks, dashed] table[x=x,y index = 1] {\thisdatafile};   
     \end{axis}
    \end{tikzpicture}
    \end{tabular}
    \figspace
    \caption{Top row: A piecewise constant signal blurred by the moving average of width $7$ and corrupted by Laplacian noise ($\sigma = \noiseLevel$), Gaussian noise ($\sigma = \noiseLevel$) or impulsive noise (with $25 \%$ of the data destroyed). 
    Bottom row: 
   The original signal (dashed line) is almost perfectly recovered by the respective minimizers of the $L^1$-Potts functional \eqref{eq:potts_L1_discFir} for all three types of noise.
    }
    \label{fig:introductory_example}
    \figspace
\end{figure}
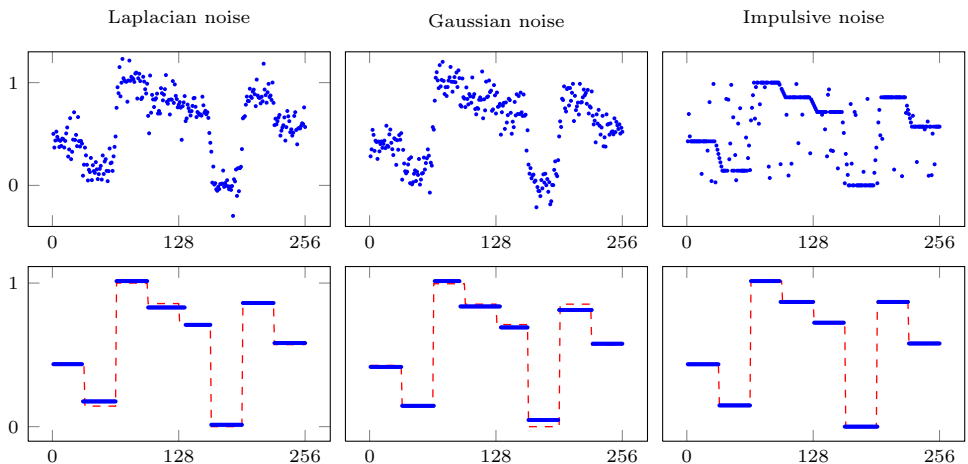

In this article, we study the \emph{$L^1$-Potts} functional 
\begin{equation}\label{eq:potts_L1_discFir} 
	P_\gamma(u) = \gamma \cdot \| \nabla u \|_0 + \|  u - f \|_1 = \gamma \cdot J(u) + \sum\nolimits_i |u_i - f_i| 
\end{equation}
and its continuous time counterpart.
Here the regularity term $J(u) = \| \nabla u \|_0$ counts the number of jumps of $u.$ 
(The abusive but appealing norm notation $\|\cdot\|_0$ has been popularized by the sparsity community \cite{elad2010sparse}.)
A minimizer of the $L^1$-Potts functional may be seen as a {\em jump-sparse} approximation to the data $f.$ 

The investigation of $L^2$-Potts functionals (with the $L^1$ norm in \eqref{eq:potts_L1_discFir} replaced by $\|\cdot\|_2^2$) can be traced back to R. Potts \cite{potts1952some} in 1952 and even to E. Ising \cite{ising1925beitrag} in 1925 in a
physical context. S. Geman and D. Geman \cite{geman1984stochastic} were the first
to consider such functionals in the context of image processing.
For a thorough account on the history we refer to \cite{winkler2005don} and references therein.
More recent results can be found in \cite{wittich2008complexity, boysen2009consistencies}.

Our investigation of $L^1$-Potts functionals is motivated by a signal processing problem; namely,
the reconstruction of piecewise constant signals from noisy and blurred spatial measurements $f$.
Recovery of piecewise constant signals from corrupted data is a task needed in various fields of applied sciences.
Examples are the reconstruction of brain stimuli \cite{winkler2005don}, the cross-hybridization of DNA \cite{drobyshev2003specificity, hupe2004analysis}
and MALDI imaging \cite{schoenmeyer2011automated}. 
The data may be incomplete frequency information as in \cite{candes2006robust, march2000reconstruction} 
or noisy and blurred measurements as in this paper. Frequently, the type of noise is unknown or known to be non-Gaussian.

The above reconstruction task is an inverse problem which is typically approached by minimization of an energy functional reflecting a tradeoff between data fidelity and regularity. 
In this regard, classical Tykhonov approaches 
favor oversmoothed solutions and are not robust to non-Gaussian noise. 
To overcome these limitations, various functionals have been proposed; cf. \cite{nikolova2003minimizers}.
In particular, $L^1$-TV and related minimization problems have gained a lot of interest in recent years, see e.g. \cite{clason2009duality,chambolle2011first,fu2006efficient, yang2009efficient} and references therein.
Under the prior assumption of jump-sparse signals, the number of jumps of $u$ is a more natural penalty than the total variation $\|\nabla u \|_1.$
This directly leads to the minimization of the $L^1$-Potts functional \eqref{eq:potts_L1_discFir}.
The numerical experiments in this paper indicate that $L^1$-Potts minimization outperforms $L^1$-TV minimization
whenever the underlying signal is jump-sparse. 
Furthermore, the
 $L^1$-Potts functional 
 is more robust to non-Gaussian noise than 
the $L^2$-Potts functional 
and yields comparable results for Gaussian noise.

The main contributions of this article are: (i) Development of a discretization framework for the continuous $L^1$-Potts functional;
(ii) Development of a fast algorithm to minimize the discretized $L^1$-Potts functionals; (iii) Proof that 
both continuous and discretized $L^1$-Potts functionals have blind deconvolution properties.

As for (i) we consider (possibly non-equidistant) samplings $f_k$ of continuous time data $f$. We show that the 
discrete $L^1$ Potts functionals associated with $f_k$ $\Gamma$-converge to the continuous $L^1$-Potts functional
as the sampling density gets finer. 
Furthermore, we show that each sequence of minimizers of the discrete functionals has an accumulation point
and that each such accumulation point is a minimizer of the continuous model. 
The precise formulation of these statements is \autoref{thm:gamma_convergence_of_discrete_potts} which is the main result of 
\autoref{sec:gamma-convergence}.
We in particular show
that the continuous Potts functional has a minimizer. (This contrasts the situation for general inverse Potts functionals,
where the continuous model need not have minimizers, but the discrete models do; cf. \cite{demaret2012deconvolution}.)  
 $\Gamma$-convergence and convergence of minimizers in relation with the $L^2$-Potts functional are topics of 
 \cite{boysen2009consistencies, boysen2007scale}. In these papers, 
$\Gamma$-convergence  is shown for a modified $L^2$-Potts functional. This approach
relies on the Hilbert space structure of $L^2$, and therefore does not directly carry over to the $L^1$ context.

Concerning (ii) we notice that the non-convexity of the jump penalty in \eqref{eq:potts_L1_discFir} makes 
standard techniques for convex optimization as used for $L^1$-TV minimization inaccessible. 
Instead, dynamic programming is used \cite{mumford1985boundary,chambolle1995image}.
A fast algorithm for $L^2$-Potts minimization has been introduced by F. Friedrich et al. and can be found 
in \cite{friedrich2008complexity}. This algorithm computes an exact minimizer in $O(n^2)$
time and $O(n)$ space, where $n$ denotes the length of the discrete data.
In their paper, 
they also consider the discrete $L^1$-Potts model.
Using a red-black tree approach they derive an $O(n^2 \log n)$ time algorithm to minimize \eqref{eq:potts_L1_discFir}
which improves the time complexity of a naive implementation by a factor of $n.$ 
However, this comes with $O(n^2)$ space consumption, thus losing a factor of $n$ 
against the naive implementation. 
V. Liebscher \cite{liebscher2011geiz} posed the question whether it is possible to minimize the $L^1$-Potts functional as fast as the classical $L^2$ version.
In this paper, we give an affirmative answer. 
We show that the $L^1$-Potts algorithm proposed computes an exact minimizer within $O(n^2)$ time and $O(n)$ space.
This statement is formulated as \autoref{thm:timeSpaceComplNonAqui} which, more generally, treats the case of not necessarily equal weights.
Such weights arise naturally in applications where data are available only on a non-uniform grid \cite{schoenmeyer2011automated}.
The key to the speed-up is to work with dynamic data structures tailored to the problem.
Besides this asymptotic result, the actual runtime of the $L^1$ algorithm is only less than $20 \%$ slower than that of the $L^2$ version, for larger data size. 
The algorithm and our results as well as numerical experiments
comparing the $L^1$-Potts algorithm with the $L^2$-Potts algorithm of \cite{friedrich2008complexity}
and the $L^1$-TV algorithm of \cite{clason2009duality} under various types of noise
are presented in \autoref{sec:alogrithm}. An efficient implementation is available for download  at \url{http://pottslab.de}.

As for (iii) we observe that besides its high robustness to noise, the  $L^1$-Potts functional has
blind deconvolution capabilities which are not shared by the $L^2$-Potts and the $L^1$-TV functional; cf. \autoref{fig:potts_deconv_small}.
In fact, we show that the continuous time $L^1$-Potts functional 
is able to exactly recover piecewise constant signals $g$ from mildly blurred measurements $K \ast g $ 
\emph{without} knowledge of the (narrowly supported) convolution kernel $K$.
Further, this property is inherited by the discretizations of the $L^1$-Potts functional.
The precise statements are formulated  as  \autoref{thm:deconvolution_main_extended} and \autoref{thm:deconvolution_main_extended_discr}.
Experiments indicate that the deconvolution property persists also in the presence of various types of noise; see \autoref{fig:introductory_example}.
Theoretical results and experiments in connection with blind deconvolution are the topic of \autoref{sec:smallSizeKernels}.

Finally, for the reconstruction of severely blurred signals (and known kernel), we consider the related deconvolution problem
\begin{align}\label{eq:potts_L1Conv}
	\gamma \cdot \| \nabla u \|_0 + \norm{K \ast u - f}_1 \to \min{}.   
\end{align}
Theoretical properties of the continuous $L^2$ variant of \eqref{eq:potts_L1Conv} have been investigated in \cite{boysen2009jump}. However, to our knowledge,  no algorithms  to solve either  \eqref{eq:potts_L1Conv} or the corresponding $L^2$-variant have been proposed yet.
In \autoref{sec:largeSizeKernels} we derive a heuristic to solve \eqref{eq:potts_L1Conv} where our fast algorithm to solve \eqref{eq:potts_L1_discFir} is employed as basic building block of an iterative algorithm; for results see \autoref{fig:potts_deconv_large}.

\section{Continuous Potts model and $\Gamma$-convergence}\label{sec:gamma-convergence}

In this section, we consider the continuous $L^1$-Potts functional defined on  $L^1[0,1]$ by
\begin{align}\label{eq:potts_continuous_L1}
	P_\gamma(u) =
		\gamma \cdot J(u) + \norm{u-f}_1, \qquad \text{if $u \in \PC$}
\end{align}
and by $P_\gamma(u) = \infty$ else. Here $\PC$ denotes the space of piecewise constant functions on the interval $[0,1].$ 
{\noop We identify a piecewise constant function with its equivalence class 
(under identification of absolutely integrable functions which are equal almost everywhere w.r.t.\ Lebesgue measure)
in order to view it as an element of $L^1[0,1].$ The resulting subspace of piecewise constant functions equals precisely 
those elements in $L^1[0,1]$ whose distributional derivatives are (finite) linear combinations of Dirac distributions centered 
in 
$(0,1).$
In view of this equivalent definition, one may also define the symbol $J(u) = \|\nabla u\|_0$ as the number of points in
the support of the distributional derivative of the element $u \in L^1[0,1].$}

We start with some basics on $L^1$-Potts functionals and
emphasize the non-unique\-ness of minimizers.
Then we have a look at discretizations of the $L^1$ Potts functional.  Our final goal is to obtain convergence statements of the discretizations to the continuous model. 
As an intermediate step, 
{\noop we investigate the convergence of semi-discrete Potts functionals.
In order to define these functionals, we consider a nested sequence $X^k \subset X^{k+1}$ of finite sets in $[0,1]$
which we call \emph{discretization sets}. 
We assume that the union of the nested sets $X^k$ is dense in $[0,1].$
(The number $n_k$ of elements in the $k$th level discretization $X^k$ is not required to equal $k,$ but we assume that $n_k$ is strictly increasing.)
We use the notation $\PCk$ to denote the space of those piecewise constant functions which only jump in the $k$th level discretization set $X^k.$
Then the sequence of semi-discrete Potts functionals $(Q_\gamma^k)$ (associated with the discretizations $X^k$) is given by
\begin{align}\label{eq:potts_semidiscrete}
	Q_\gamma^k(u) =
		\gamma \cdot J(u) + \norm{u-f}_1, \qquad \text{if $u \in \PCk$} 
\end{align}
and by $Q^k_\gamma(u) = \infty$ else.} 
The auxiliary functionals $(Q_\gamma^k)$ are still defined by data $f$ defined {\noop on the continuum $[0,1].$} 
However, in practice, data are given by discrete samples 
$	f^k = S_k f $
obtained from $f$ by a sampling operator $S_k$ at some level $k.$ 
One implementation of $S_k$ might be the integral sampling
\begin{align}\label{eq:sampling_integral}
	f^k(j) = S_k f(j) =   \tfrac{1}{|I_j|}  \int_{I_j} f \de \lambda,  
\end{align}
where $I_j$ are the intervals given by the discretization sets $X^k.$
If $f$ is continuous, one might also consider point sampling
\begin{align}\label{eq:sampling_point}
	f^k(j) = S_k f(j) = f(x_j),  
\end{align}
where $x_j$ might be taken as the midpoint of the interval $I_j.$
Then, the discrete Potts functionals read
	\begin{align} \label{eq:potts_fully_discrete}
	P_\gamma^k(u) =
		\gamma \cdot J(u) + \|u-\sum\nolimits_j S_k f(j) \cdot \indfunc_{I_j}\|_1, \qquad \text{if $u \in \PCk$},
\end{align}
and $P^k_\gamma(u) = \infty$ else.

In \autoref{ssec:gamma_convergence_semi} we obtain convergence statements (\autoref{thm:gamma_convergence_semi}) for these semi-discrete Potts functionals
which we then apply in \autoref{ssec:gamma_convergence} to obtain the corresponding statements for the (fully) discretized $L^1$ Potts functionals 
we are ultimately interested in. The main result is \autoref{thm:gamma_convergence_of_discrete_potts}.

\subsection{Basics on  $L^1$-Potts functionals}\label{ssec:basics_of_potts}

It is well known that a median $\mu$ minimizes the $L^1$ error for a constant approximation,
that is, 
\begin{align}\label{eq:minimizerPropertyMedian}
	\| g- \mu \|_1 \leq \|g-y\|_1, \quad\text{for all $y \in \R.$}
\end{align}
Thus, in order to find a minimizer $u$ of the $L^1$-Potts problem, 
it is sufficient to find its jump locations  $x_i.$
Then,  a corresponding minimizer is given by $u = \sum_{i} \mu_i \indfunc_{(x_{i}, x_{i+1})},$
where $\mu_i$ is a median of the data $f$ on the interval $(x_i, x_{i+1}).$
Recall that a \emph{median} $\mu$ of a real-valued function $f$
 with respect to some finite measure $\nu$ on $\Omega$ is characterized by the property
\begin{align}\label{eq:median_set}
  \nu (\{x : f(x) < \mu \})
    \leq \tfrac{1}{2} \nu(\Omega) \quad  \text{ and } \quad \nu (\left\{x : f(x) >  \mu \right\})   \leq  \tfrac{1}{2} \nu(\Omega).
\end{align}
{\noop In our case, we use weighted counting measures in the discrete time case and
 the Lebesgue measure on the interval in the continuous time case.
If $f$ is a continuous function and $\nu$ is the Lebesgue measure then there is only one unique median, whereas in general, 
there might be a whole interval of medians.} 

This non-uniqueness of the median is one source of non-uniqueness of minimizers of the (discrete and continuous) $L^1$-Potts functional.
Besides this, there are two other sources of non-uniqueness as  the following example shows.

\begin{example}\label{ex:nonuniqueness}
Consider the data $f = \indfunc_{[0,\frac14]} + \indfunc_{[\frac12,\frac 34]}$
on the interval $[0,1].$ There are three sources for non-uniqueness of the minimizers of the
$L^1$-Potts functional $P_\gamma$ associated with $f$ depending on the parameter $\gamma:$ 
First, for $\gamma > \frac{1}2,$ a minimizer of $P_\gamma$ is constant on $[0,1],$
any constant candidate $u_a(x) = a,$ $a \in [0,1],$  has  Potts value $P_\gamma =\frac{1}2.$
Here the reason for a non-unique minimizer is the non-uniqueness of the median.
Second, for $\gamma = \frac12,$ 
we easily check that the three-jump solution $u=f$ 
and the one-jump solution $u=\indfunc_{[0,\frac14]}$ attain the same minimal value.
Thus in that case, there is non-uniqueness with respect to the number of jumps.
Third, if $ \frac18 < \gamma < \frac14,$ 
both  $u_1 = \indfunc_{[0,\frac14]}$ and $u_2 = \indfunc_{[0,\frac34]}$
are minimizers of the $L^1$-Potts functional.
Here, the number of jumps is equal, but the minimizer is not unique with respect to the partitions.
{\noop (For completeness:  for $\gamma < \frac{1}{8}$ the minimizer coincides with data $f;$ for $\gamma = \frac{1}{8},$ possible minimizers are data $f$ as well as the one jump solutions $u_1$ and $u_2.$)}
\end{example}

Note that the first type of non-uniqueness is specific to the $L^1$-Potts functional. Indeed, since the mean value is unique, this kind of non-uniqueness does not occur for the $L^2$-Potts functional.
The other two types of non-uniqueness also arise in the $L^2$ case
but have been shown to occur, in the discrete setting, only on a negligible set of data $f$ and parameters $\gamma$
\cite{wittich2008complexity}. We strongly conjecture that analogous statements (modulo the non uniqueness of medians) can
be made for the $L^1$-Potts functional. However, a rigorous proof is out of the scope of this article.

\subsection{$\Gamma$ convergence of semi-discrete Potts functionals}\label{ssec:gamma_convergence_semi}

We recall that a sequence $(F^k)$ of functionals on $L^1[0,1]$
taking values in the extended positive real line $[0, \infty],$
is \emph{$\Gamma$-convergent} to a functional $F^\infty$ if 
\begin{enumerate}[(i)]
	\item for each $u \in L^1[0,1]$ and each sequence $(u^k)_{k \in \N}$ in $L^1[0,1]$
	with $u^k \to u$ as $k \to \infty$ holds 
	\begin{align}\label{eq:gamma_conv_1}
		F^\infty(u) \leq \liminf\nolimits_{k\to \infty} F^k(u^k),
	\end{align}
	\item and for each $u \in L^1[0,1]$ there is a \emph{recovery sequence} $ (u^k)_{k\in \N}$ in $L^1[0,1]$
	with $u^k \to u$ as $k \to \infty$ such that
	\begin{align}\label{eq:gamma_conv_2}
		F^\infty(u) \geq \limsup\nolimits_{k\to \infty} F^k(u^k).
	\end{align}
\end{enumerate}
{\noop The fundamental theorem of $\Gamma$-convergence is as follows.
\begin{theorem} \label{thm:fundGammaCon}
Let $(F^k)$ be a sequence of functions on a metric space $X$ with values on the extended real line.
Assume that $(F^k)$ is $\Gamma$-convergent to $F^\infty.$ 
If there is a compact set $K \subset X$ such that $\inf_X F^k = \inf_K F^k$ for all k, 
then $F^\infty$ has a minimizer and a minimizing value is given by $\lim\nolimits_k \inf\nolimits_X F^k.$
Moreover, if $(u^k)$ is a precompact sequence such that $\lim_k F^k(u^k) = \lim_k \inf_X F^k,$ then every limit of a subsequence of $(u^k)$ is a minimum point for $F^\infty.$
\end{theorem}

For a proof and further information see e.g. \cite{braides2002gamma}.
Our first goal is to show \eqref{eq:gamma_conv_1} for the semi-discrete Potts functionals $Q_\gamma^k.$
To this end we need some preparatory lemmas on piecewise constant functions.}

\begin{lemma} \label{lem:CompactnessOfMCj}
The set of piecewise constant functions which have at most $j$ jumps and which are bounded by the constant $C$ (w.r.t. the $\sup$-norm)
\begin{align}\label{eq:compactness}
		M_{C,j} = \left\{ g = \sum\nolimits_{l=1}^{j+1} a_l \, \indfunc_{(x_{l-1}, x_{l})} :  0 = x_0  \leq \ldots \leq x_{j+1} = 1, |a_l| \leq C \right\}
\end{align}
is a compact subset of $L^1[0,1]$. 	
\end{lemma}

\begin{proof}
We have that $M_{C,j}$ is the image of the compact set   
$$
   A= \left\{ (x,a) \in \R^{j+2} \times \R^{j+1} \; : \;  0 = x_0  \leq \ldots \leq x_{j+1} = 1, |a_l| \leq C \right\} 
$$
{\noop under the mapping $e(x,a) = \sum_{l=1}^{j+1} a_l \, \indfunc_{(x_{l-1}, x_{l})}$ which is a continuous mapping into $L^1[0,1]$ for the following reason. 
We fix $(x,a) \in A$ and consider a small perturbation $(x+\varepsilon_1,a+\varepsilon_2)$ in $A.$ 
(In the following we always consider representatives of the $L^1$ elements $e(x,a)$ which are constant between the $x_l$.)
If $|e(x,a)(t)- e(x+\varepsilon_1,a+\varepsilon_2)(t)| > \|\varepsilon_2\|_{\infty},$ then $t$ must be near a jump, i.e.,
there is an index $i$ such that $t$ is contained in the interval $I_i = [\min(x_i,x_i+\varepsilon_{1,i}), \max(x_i,x_i+\varepsilon_{1,i})].$
The Lebesgue measure of the union of these intervals $I_i$ equals $\|\varepsilon_1\|_1$ and since $e(x,a),e(x+\varepsilon_1,a+\varepsilon_2) \in A$
their absolute value does not exceed $C.$ Therefore, we have
\begin{align*}
   \|e(x &+\varepsilon_1,a+\varepsilon_2)-e(x,a)\|_1 
   = \|e(x+\varepsilon_1,a+\varepsilon_2)-e(x,a)|_{\cup I_i}\|_1 \\
   &+ \|e(x+\varepsilon_1,a+\varepsilon_2)-e(x,a)|_{[0,1]\setminus \cup I_i}\|_1  
   \leq 2C \|\varepsilon_1\|_1 + \|\varepsilon_2\|_\infty
\end{align*}
which implies that $e$ is continuous.
(In principle, we showed that the mappings $(x_i,x_{i+1}) \mapsto \indfunc_{[x_i,x_{i+1}]} $ 
are continuous and use the continuity of the vector space operations.)
Hence, $M_{C,j}$ is compact.}
\end{proof}

\begin{lemma} \label{lem:number_of_jumps}
Let $u$ be a piecewise constant function with minimal 
jump height $h$ and minimal interval width $l.$
Then, for any piecewise constant function $w,$ we have the implication	
\begin{align}\label{eq:number_of_jumps}
		J(w) < J(u) \qquad \Longrightarrow \qquad \| w - u\|_1 \geq 
		 h \cdot l.
\end{align}
\end{lemma}

\begin{proof}
   {\noop We first show that it is sufficient to prove \eqref{eq:number_of_jumps} for functions $w$
	 whose jump set is included in the jump set of $u.$
	 To this end, we consider} 
	 a general $w$ with less jumps than $u$
	 and shift its jump locations to those of $u$
	 such that the resulting $w'$ fulfills $J(w) \geq J(w')$
	 {\noop and $ \| w - u \|_1 \geq \| w' - u \|_1.$
	 In the following we construct such a function $w'.$}
	 Let $X = \{ x_1,..., x_{n} \}$ be the jump set of $u = \sum_{i} a_i \indfunc_{(x_i, x_{i+1})}$
	 and let  $T = \{ t_1,..., t_r\}$ be the jump set of $w.$
	 Denote $S = T \setminus X$ the jump locations of $w$ which
	 are not in the jump locations of $u.$
	 We reduce $S$ to $\emptyset$ by the following recursive procedure.
	 If $S = \emptyset$ we are done. 
	 If $S \neq \emptyset,$  we pick a jump point $t_i \in S.$  
	 Then we consider the difference between $u$ and 
	 $w= \sum_j b_j \indfunc_{(t_j, t_{j+1})}$ to the left of
	 $t_i,$  i.e., $|u(t_i) - b_{i-1}|$ 
	 and the difference to the right $|u(t_i) - b_i|.$
	 If  $|u(t_i) - b_{i-1}| \geq |u(t_i) - b_i|$ 
	 then we shift $t_i$ to the left until we meet the next jump point $p \in X \cup T.$
	 {\noop(If there is no such point, we shift $t_i$ to $0$ and let $p=0.$)}	 
	 If the point $p$ is in $X \setminus T,$ say $p$ equals $x_l,$ 
	 then $w'$ reads as 
	 \begin{align*}
	 	w' = \sum\nolimits_{j \neq {i-1, i}} b_j  \indfunc_{(t_j, t_{j+1})} + 
		b_{i-1} \indfunc_{(t_{i-1}, x_{l})} + b_i \indfunc_{(x_l, t_{i+1})}.
	 \end{align*}
	 We have $J(w') = J(w),$  $\| w - u\|_1 \geq \| w' - u\|_1,$
	 and $w'$ has one more  jump in the jump set of $u$ than $w.$ 
	 If $p$ is in $T$ or $p=0,$ the procedure reduces the number of jumps
	 of $w$ by $1$ and still $\| w - u\|_1 \geq \| w' - u\|_1.$
	 In any case, the number of jumps which are not in $X$
	  is reduced by one.
	  In the opposite case,  $|u(t_i) - b_{i-1}| < |u(t_i) - b_i|,$
	  we shift to the right and  proceed analogously 
	  {\noop (where we shift to $1$ and let $p=1$ 
	  whenever there is no jump of $u$ or $w$ being further to the right).}
	 Iterating this algorithm, we obtain a $w'$ whose jump set is contained in that of $u.$
	 By construction, this $w'$ still has less jumps than $u,$ i.e., $J(w') < J(u),$ and 
	  $\| w - u\|_1 \geq \| w' - u\|_1.$ 
	 
	 In consequence, we find an interval $(t_i, t_{i+1})$
	 where $w'$ does not jump
 but $u$ does so at $t_i < p < t_{i+1}.$
 Then 
	 \begin{align}
	 	\| w' - u\|_1 \geq | u(p^+) - u(p^-) | \cdot \min(p-t_i, t_{i+1} - p) 
		   \geq 
		   h l,
	 \end{align}
	which shows \eqref{eq:number_of_jumps}.	 
\end{proof}

{\noop We now show the first condition for $\Gamma$-convergence \eqref{eq:gamma_conv_1} for the semi-discrete Potts functionals $(Q_\gamma^k).$ }
\begin{lemma}\label{lem:lower_semicontiuous}
	Let $u \in L^1[0,1]$ and let $(u^k)$ be a sequence in $L^1[0,1]$ such that $\|u^k - u\|_1 \to 0$ for $k \to \infty.$
	Then 
		$P_\gamma(u)$ $\leq \liminf\nolimits_{k\to \infty} Q_\gamma^k(u^k).$ 
\end{lemma}
\begin{proof}
We may restrict ourselves to sequences $(u^k)$ in $\PC.$ We first consider $u \in L^1$ which is not essentially bounded, i.e., $u \not\in L^\infty.$
We show that, for each sequence of piecewise constant functions $(u^k)$ converging to $u,$ 
the number of jumps tends to infinity, which then yields Lemma \ref{lem:lower_semicontiuous} for essentially unbounded $u$. 
To this end, we construct a sequence of disjoint intervals $(I_i)$ 
and a sequence of positive numbers $(\varepsilon_i)$ for which we show the following:
{\noop for an approximating piecewise constant function $\tilde{u},
$ $\|\tilde{u}-u\|_1 < \varepsilon_i$  implies that $|\tilde{u}|$ takes values in each of the disjoint sets $I_1,\ldots,I_i.$}
Since $u$ is essentially unbounded, the Lebesgue measure $\lambda(\{|u| \geq n\})$ is positive for any positive integer $n.$
We choose $n_1= 1$ and 
{\noop let $\varepsilon_1 = \lambda(\{|u| \geq 2 n_1\})/2.$} 
Using the Markov inequality, we find $m_1 > 2n_1$ such that
$\lambda(\{|u| > m_1\}) \leq \varepsilon_1.$ This implies $\lambda(\{ 2 n_1 \leq |u| \leq m_1\})$ $\geq \varepsilon_1.$
If $|\tilde{u}|$ does not take a value 
in the interval $I_1 = [n_1,2m_1]$ $\cup [-2m_1,-n_1],$
then we can estimate its approximation error by
$$
\|\tilde{u}-u\|_1 \geq \|(\tilde{u}-u)|_{\{2n_1 \leq |u| \leq m_1 \}}\|_1 \geq n_1 \cdot \lambda(\{2n_1 \leq |u| \leq m_1 \}) 
\geq \varepsilon_1. 
$$
We choose $n_2 = 2m_1 +1$ and let $\varepsilon_2 = \lambda(\{|u|> 2 n_2\})/2.$ We proceed likewise to obtain intervals $I_i$ and
positive numbers $\varepsilon_i$ with the claimed properties. Then, if $(u^k)$ converges to $u,$ there is some index $k_i$ such
that $\|u^{l} - u\|_1 < \varepsilon_i$ for all $l > k_i.$ Each such piecewise constant function $u^l$ must take values in 
each of the intervals $I_1, \ldots, I_i.$ Thus the number of jumps of the sequence $(u^k)$ tends to infinity.

Next, we consider essentially bounded $u.$ We show that, if $(u^k)$ tends to $u$ and the number of jumps of the $u^k$ is (uniformly) bounded,
then $u$ is piecewise constant (i.e., it has a piecewise constant representative). 
{\noop This implies that for non piecewise constant $u,$ 
the number of jumps of a sequence $(u^k)$ converging towards $u$ must go to infinity for the following reason.  
If there was a subsequence of $u^k$ with a uniformly bounded number of jumps, its limit would be piecewise constant.
Furthermore, the limit would equal $u$ and so $u$ would be piecewise constant which would be a contradiction.
Hence, in order to show Lemma \ref{lem:lower_semicontiuous} for non piecewise constant essentially bounded functions 
we have to show that the uniform boundedness of the number of jumps of a converging sequence $(u^k)$ implies the piecewise constance of its limit.
}

In a first step, we construct a sequence $(\tilde{u}^k)$
which converges to $u$ and which is uniformly bounded (w.r.t. the $\sup$-norm) as follows.
We define the piecewise constant function $\tilde{u}^k$ by $u^k $ except for the intervals on which $|u^k|> 2 \|u\|_\infty.$ 
Here we let $\tilde{u}^k = 0.$ Then $(\tilde{u}^k)$ is uniformly bounded by $2 \|u\|_\infty$ and 
$\|\tilde{u}^k-u\|_1  \leq \|u^k-u\|_1.$ This entails that $(\tilde{u}^k)$ converges to $u.$
So we may assume that $(u^k)$ tends to $u$ in $L^1,$ that the piecewise constant functions $(u^k)$ are uniformly bounded (w.r.t. the $\sup$-norm)
and that the number of jumps of the $(u^k)$ is bounded as well. 
{\noop Then the piecewise constance of $u$ is a consequence of the compactness of the sets $M_{C,j}$ defined by \eqref{eq:compactness}  
in Lemma \ref{lem:CompactnessOfMCj}.}

So far, we have shown Lemma \ref{lem:lower_semicontiuous} for non piecewise constant $u.$ 
Now we consider a piecewise constant function $u$ and show the assertion in that case.
As a consequence of \eqref{eq:number_of_jumps} in Lemma \ref{lem:number_of_jumps}
we find for each sequence of piecewise constant functions $(u^k)$ converging to $u$ in $L^1$
 an index $k_0$ such that, for all $k \geq k_0,$
	\begin{align}\label{eq:jump_estimate}
		J(u^k) \geq J(u).
	\end{align}
	Equipped with this estimate we show
	the assertion of the lemma for piecewise constant functions $u$.
	Since $Q^k_\gamma(v) = \infty$ for  $v \notin \PCk$
	we may assume without loss of generality that $u^k$ is in $\PCk.$	
	Then $P_\gamma(u^k) = Q_\gamma^k(u^k)$
	and thus, using \eqref{eq:jump_estimate}, 
	\begin{align} \notag
		P_\gamma(u)  = \gamma J(u) + \|u-f\|_1 &\leq \liminf_k \gamma J(u^k) + \lim_k \|u^k-f\|_1 \\ 
		&= \liminf_k P_\gamma(u^k) = \liminf_k Q^k_\gamma(u^k)
		\label{eq:lsc}
	\end{align}
	which completes the proof.
\end{proof}

Next we show the second condition for $\Gamma$-convergence \eqref{eq:gamma_conv_2}.
\begin{lemma}\label{lem:recovery_sequence}
	For each $u$ in $L^1[0,1]$ there is a recovery sequence $(u^k)$ which converges to $u$ in $L^1$ 
	and which fulfills
		$P_\gamma(u)$ $\geq \limsup\nolimits_{k\to \infty} Q_\gamma^k(u^k).$ 
\end{lemma}
\begin{proof}
    We may restrict  $u$ to $\PC,$ since otherwise $P_\gamma(u) = \infty$.
		For $k \in \N$ we define 
	 $\epsilon_k$ as the largest gap in $X^k,$
	 i.e.,
	  $\epsilon_k = \sup_l |x_l - x_{l+1}|,$
	 where the $x_l$ are a monotone ordering of the points in $X^k$ (including $0$ and $1$). 
	Since the union of discretization sets $X^k$ is dense in $[0,1],$
 the sequence $(\epsilon_k)$ converges to $0$ as $k \to \infty.$  
	  
	  Let us fix $k$ for the moment.
	  For every jump location $t_j$ of $u$ we find 
	  a closest discretization point $x_j$ such that $|x_j - t_j| < \epsilon_k;$ 
	 Then, for sufficiently large $k$, replacing each interval $(t_i, t_{i+1})$
	 in $u = \sum_i a_i \indfunc_{(t_i, t_{i+1})}$
	 by the interval $(x_i, x_{i+1}),$ we obtain a piecewise constant function
	 $u^k$ such that 
	 \begin{align*}
	 	\| u - u^{k} \|_1 \leq \epsilon_k \cdot H \cdot J(u).
	 \end{align*}
	 Here, $H$ is the maximal jump height of $u.$
	 For sufficiently large $k,$
	 $u$ and $u^{k}$ have the same number of jumps, since $(\epsilon_k)$ converges to $0$ 
	 and so each jump point of $u$ is assigned exactly one point in $X^k.$ 
	 Therefore, we have
	 \begin{align*}
	 	P_\gamma(u) = \gamma J(u) + \| u - f\|_1
		&\geq \gamma J(u^{k}) +  \| f - u^{k}\|_1 - \| u^{k} - u \|_1 \\
		&\geq P_\gamma(u^{k}) - \epsilon_k \cdot H \cdot J(u).
	 \end{align*}
	 Rewriting this as
	$ P_\gamma(u) + \epsilon_k \cdot H \cdot J(u)$ $\geq P_\gamma(u^{k})$ $= Q^k_\gamma(u^{k})$
	 and passing to the limit $k \to \infty$
	 yields the assertion.
\end{proof}

{\noop The following lemma is an auxiliary statement needed for the proof of the next theorem which is our main result on the semi-discrete Potts functionals $(Q^k_\gamma).$
It is also needed in the proof of \autoref{thm:gamma_convergence_of_discrete_potts}.}
\begin{lemma}\label{lem:CombMarkAbsCont}
{\noop   Let $f \in L^1[0,1]$ and let $f^k,$ $k=1,2,\ldots$ be the functions obtained from $f$ via the integral sampling \eqref{eq:sampling_integral}.
	 For $\varepsilon>0$ there is a constant $C>0$ such that, for any $g \in \{f,f^1,f^2,\ldots\},$  
\begin{align} \label{eq:boundOnIntervallengthImpliesSmallIntegralA} 
		 |\mu_I(g)| \geq C     \qquad \implies \qquad  \int\nolimits_I |g| \leq \varepsilon.
\end{align}
	 Here $\mu_I(g)$ denotes a median of $g$ on the interval $I.$ 
	 The interval $I$ is arbitrary for $g=f,$ and for $g=f^k,$ we assume that the endpoints of $I$ are in the discretization set $X^k.$ }
\end{lemma}

\begin{proof}
{ \noop
For $h \in L^1,$ we have the following connection between the absolute value of a median $\mu_I(h)$ on some interval $I$ and the length of the interval $|I|$
\begin{align} \label{eq:boundOnIntervallengthA}
		 |\mu_I(h)| \geq c     \qquad \implies \qquad  |I| \leq \tfrac{2\|h\|_1}{c}       \qquad \text{for any } c>0.
\end{align} 
Indeed,  $|\mu_I(h)| \geq c$ implies that $\lambda(\{|h| \geq c \}) \geq |I|/2,$ and the Markov inequality
yields $\lambda(\{|h| \geq c \}) \leq \|h\|_1/c.$ We apply \eqref{eq:boundOnIntervallengthA} to $f^k$
and note that since $f^k = \sum_l S_k f(l) \cdot \indfunc_{I_l}$ we have $\|f^k\|_1 \leq \|f\|_1.$
We get that, for any interval $I$, 
 \begin{align} \label{eq:boundOnIntervallengthDiscrA}
		 |\mu_I(f^k)| \geq c     \qquad \implies \qquad  |I| \leq \tfrac{2\|f\|_1}{c}       \qquad \text{for any } c>0.
\end{align}
Furthermore, if the interval length $|I|$ is small, the integral 
of $|f|$ on $I$ is small, i.e., for $\varepsilon>0$ there is $\delta>0$ such that 
\begin{align} \label{eq:boundOnIntegralA}
		 \textstyle |I| \leq \delta     \qquad \implies \qquad  \int\nolimits_I |f| \leq \varepsilon.
\end{align} 
This implication can be seen by considering the finite measure $|f| d\lambda$ with density $|f|$ with respect to the
Lebesgue measure $\lambda.$ The corresponding distribution function is absolutely continuous and thus uniformly continuous
which is precisely the statement of \eqref{eq:boundOnIntegralA}.
Since $\int\nolimits_I |f^k|$ $\leq \int\nolimits_I |f|,$ for all intervals $I$ whose endpoints are in the discretization set $X^k,$
we get that \eqref{eq:boundOnIntegralA} implies
\begin{align} \label{eq:boundOnIntegralDiscrA}
		 \textstyle |I| \leq \delta     \qquad \implies \qquad  \int\nolimits_I |f^k| \leq \varepsilon
\end{align} 
for all such intervals $I.$ 
Combining \eqref{eq:boundOnIntervallengthDiscrA} and \eqref{eq:boundOnIntegralDiscrA} yields the assertion of the lemma for $g=f^k,$
and combining \eqref{eq:boundOnIntervallengthA} for $h=f$ and \eqref{eq:boundOnIntegralA} yields the assertion for $g=f.$}
\end{proof}

We formulate our main result on the semi-discrete Potts functionals $(Q^k_\gamma).$

\begin{theorem}\label{thm:gamma_convergence_semi}
Let $f$ be an integrable function on $[0,1]$ and let $P_\gamma$  be the corresponding
continuous $L^1$-Potts functional. Then the semi-discrete Potts functionals $(Q_\gamma^k)$ with respect to a nested sequence of 
discretization  sets $(X^k)$ $\Gamma$-converge to $P_\gamma.$ 
Furthermore, each functional $Q_\gamma^k$ has a minimizer. 
Each sequence $(u^k),$ where $u^k$ is a minimizer of $Q_\gamma^k,$ has at least one accumulation point. 
Each such accumulation point $u^\ast$ is a minimizer of $P_\gamma$, i.e.,	
	$	P_\gamma(u^\ast) = \inf\nolimits_{v \in  L^1} P_\gamma(v).$ 
\end{theorem}

{\noop
As an immediate consequence we have the following statement.
\begin{theorem}
The continuous $L^1$-Potts functional $P_\gamma$ has a minimizer.
\end{theorem}
}

{\em Proof of \autoref{thm:gamma_convergence_semi}.}
Our first aim is to show that all the minimizers of the functionals $(Q_\gamma^k)$ are contained in a compact subset $M$ of $L^1[0,1]$
which is independent of $k$. We let $M=M_{C,j}$ be the compact set from \eqref{eq:compactness}, where
we define the maximal number of jumps $j$ by the smallest integer larger than $\|f\|_1/\gamma$.
The constant $C$ is defined in the course of the proof. 
To show that all the minimizers are contained in $M,$ 
we first prove that 
\begin{align} \label{eq:strong_equi_mildly}
    \inf_{u \in M} Q_\gamma^k(u) = \inf_{u \in \PC} Q_\gamma^k(u) < Q_\gamma^k(v)  \text{ for all } v \notin M.
\end{align}
To this end,
we consider an arbitrary piecewise constant function $u \notin M$ and show that there is a piecewise constant function $\tilde{u}$ contained in $M$ with lower Potts value.
 Because of the minimizing property \eqref{eq:minimizerPropertyMedian} of the median,  
 we may assume that $u = \sum_{i} \alpha_i \indfunc_{(x_{i}, x_{i+1})}$ is given by a median $\mu_I$ 
 on each interval $I = (x_{i}, x_{i+1}).$
{\noop If the number of jumps of $u$ obeys $J(u) > j > \|f\|_1/\gamma$ then $\tilde{u} = 0$ has a smaller Potts functional value since 
\begin{align} \label{eq:large_jumps_number}
P_\gamma(u) \geq \gamma J(u) > \|f\|_1 \geq P_\gamma(0) \geq \inf\nolimits_{v \in\PC}  P_\gamma(v).
\end{align}
Therefore, we may assume that $u$ has at most $j$ jumps.
Now we define the constant $C$ in $M_{C,j}$ as the resulting number $C$ in \eqref{eq:boundOnIntervallengthImpliesSmallIntegralA} for $g=f$
and the choice $\varepsilon = \gamma/6.$ This means
\begin{align} \label{eq:boundOnIntervallengthImpliesSmallIntegral}
		 |\mu_I(f)| \geq C     \qquad \implies \qquad  \int\nolimits_I |f| \leq \gamma/6 \qquad \text{for any interval } I.
\end{align}
If $C< 3 \|f\|_1$ we enlarge it such that $C \geq 3 \|f\|_1.$}
We observe that there is at least one interval $I = (x_{i}, x_{i+1})$
where $|u| \leq C.$ If this was not the case, we would have $\lambda(\{|f|>C \}) \geq 1/2.$ 
This contradicts the assumption  $C \geq 3 \|f\|_1$
since $\lambda(\{|f|>C \}) \leq \|f\|_1/C \leq 1/3.$
We start with an interval $I$ where $|u| \leq C$ and define $\tilde{u} = u$ on $I.$ 
Then we consider a (left or right) neighbor interval $J.$ 
{\noop (Such a neighboring interval always exists since we assumed that $u$ is not in $M_{C,j}.$) }
If $|u|\leq C$ on $J,$
we define $\tilde{u}=u$ on $J$ and proceed with a neighbor of $I$ or $J.$
If $|u| > C$ on $J,$ we define $\tilde{u}|_J = u|_I,$ i.e., we remove a jump.
We have $J(\tilde{u}) = J(u)-1.$ 
Furthermore, $\|\tilde{u} |_J\|_1$ $\leq \|u|_J\|_1,$ since $|u|_I| \leq C < |u|_J|.$  
Also, $\|u|_J\|_1$ $\leq 2 \|f|_J\|_1$ because the value of $u$ on $J$ is a median of $f$ on $J$
and thus $ \|u|_J\|_1 - \|f|_J\|_1$ $\leq \|(f-u)|_J\|_1$ $\leq \|f|_J\|_1.$ 
By \eqref{eq:boundOnIntervallengthImpliesSmallIntegral},
$\|(\tilde{u}-f )|_J\|_1$ $\leq \|\tilde{u} |_J\|_1 $ $+  \|f|_J\|_1 $
$\leq 3 \|f|_J\|_1$ $\leq \gamma/2.$ In consequence,
\begin{align} \label{eq:betterPottsValue}
P_\gamma(\tilde{u}) = \gamma J(\tilde{u}) + \|f-\tilde{u}\|_1 \leq \gamma J(u) - \gamma + \|f-u\|_1 + \gamma/2  < P_\gamma(u).
\end{align}   
Proceeding likewise for the other intervals, we obtain a piecewise constant function $\tilde{u}$ which is smaller than $C$ in absolute
value and thus is contained in $M_{C,j}.$ Furthermore, $P_\gamma(\tilde{u}) < P_\gamma({u}).$
Since $P_\gamma = Q^k_\gamma$ for piecewise constant functions which only jump in the discretization set $X^k$
and since the jump set of $\tilde{u}$ is contained in the jump set of $u,$ \eqref{eq:betterPottsValue}
remains true when we replace $P_\gamma$ by $Q^k_\gamma.$ This shows \eqref{eq:strong_equi_mildly}.

We are now able to prove the assertions of the theorem. Lemma \ref{lem:lower_semicontiuous} and Lemma \ref{lem:recovery_sequence} imply the $\Gamma$-convergence of $(Q_\gamma^k)$  to $P_\gamma$. 
Since \eqref{eq:strong_equi_mildly} holds, we deduce from the lower semicontinuity of $Q_\gamma^k$ 
(which is a consequence of \eqref{eq:lsc})
and the compactness of $M$ that each $Q_\gamma^k$ has a minimizer which then is in $M$. 
Furthermore,  
the compactness of $M$ implies the existence of a cluster point for any  sequence of minimizers $(u^k)$.
Then the last assertion of the theorem is a consequence of the fundamental theorem of $\Gamma$-convergence;   
cf. \autoref{thm:fundGammaCon} or \cite[Theorem 1.21, p. 29]{braides2002gamma}.
\quad \endproof

\subsection{$\Gamma$ convergence of discrete Potts functionals}\label{ssec:gamma_convergence}
We use the results on the semi-discrete Potts functional to obtain 
$\Gamma$-convergence with respect to the $L^1$-norm for the discrete Potts functionals defined by \eqref{eq:potts_fully_discrete}. 
\begin{theorem}\label{thm:gamma_convergence_of_discrete_potts}
Let $f$ be an integrable function and let $P_\gamma$  be the corresponding continuous $L^1$-Potts functional.
Then the discrete $L^1$-Potts functionals $(P_\gamma^k)$ 
$\Gamma$-converge to $P_\gamma.$
Each sequence $(u^k),$ where $u^k$ is a minimizer of $P_\gamma^k,$ has at least one accumulation point. 
Each such accumulation point $u^\ast$ is a minimizer of $P_\gamma$, i.e.,
	$	P_\gamma(u^\ast) = \inf\nolimits_{v \in  L^1[0,1]} P_\gamma(v). $ 
\end{theorem}

{\em Proof.}
  We start showing the statement on $\Gamma$-convergence. We first establish a quantitative relation between
	$P_\gamma^k$ and $Q_\gamma^k.$ We may estimate, {\noop for any $u \in \PCk,$} 
  \begin{align}\label{eq:relation_P_Q}
	P_\gamma^k(u) 
	&= \gamma \cdot J(u) + \| u- f \|_1 \notag\\
	&\leq \gamma \cdot J(u) + \| u-\sum\nolimits_j S_k f(j) \cdot \indfunc_{I_j} \|_1  +  
	    \| f-\sum\nolimits_j S_k f(j) \cdot \indfunc_{I_j}  \|_1 \notag\\
	&= Q_\gamma^k(u)  +  \| f-\sum\nolimits_j S_k f(j) \cdot \indfunc_{I_j}  \|_1.
	\end{align}
	If $S_k$ is given by the integral sampling \eqref{eq:sampling_integral} we have the following  inequality 
	\begin{align}\label{eq:difference_estimate_by_modulus_of_cont}
	\| f-\sum\nolimits_j S_k f(j) \cdot \indfunc_{I_j}  \|_1 \leq C \cdot \omega(f,\sup\nolimits_j |I_j|)
	\end{align}
	where $C$ is a positive constant and $\omega$ is the $L^1$ modulus of continuity given by
	$	\omega(f, t)$ $= \sup_{|h| \leq t} \| f(\argdot + h) - f\|_1,$ cf. \cite{devore1993constructive}.
		Since the translation is continuous in $L^1$ and 
	the union of the nested sampling sets $X^k$ is dense,
	it follows by \eqref{eq:difference_estimate_by_modulus_of_cont} that  
	$		\| f-\sum_j S_k f(j) \cdot \indfunc_{I_j}  \|_1 \to 0,$ as $k \to \infty.$	
	Now exchanging the roles of $P_\gamma^k$ and $Q_\gamma^k$ in \eqref{eq:relation_P_Q}
	we conclude that
	\begin{align}\label{eq:diffPandQ}
		| P_\gamma^k(u) - Q_\gamma^k(u) | \leq 
		\| f-\sum\nolimits_j S_k f(j) \cdot \indfunc_{I_j}  \|_1 \to 0, \qquad\text{as $k \to \infty.$}
	\end{align}
	We emphasize that the above convergence is uniform in $u.$ 
	If we consider point sampling \eqref{eq:sampling_point}
	and a continuous function $f$, \eqref{eq:diffPandQ} remains true.		
		This is because
		\begin{align}
		\| f - \sum\nolimits_j S_k f(j) \indfunc_{I_j}\|_1 \leq \sup_j \sup_{x,y \in I_j} |f(x) - f(y)|
		\end{align}
		tends to zero as $k$ increases since $f$ is uniformly continuous on $[0,1]$ 
		and the maximum interval length $|I_j|$ goes to zero since the union of the nested $X^k$ is dense.
		
		In order to show $\Gamma$-convergence, we show \eqref{eq:gamma_conv_1} and \eqref{eq:gamma_conv_2} 
		for the discrete Potts functionals.
		Let us consider a sequence $(u^k)$ which converges to $u$ in $L^1.$ Then 
		{\noop $Q_\gamma^k(u^k)$ $\leq P_\gamma^k(u^k)$ $+|P_\gamma^k(u^k) - Q_\gamma^k(u^k)|$.} 
		This, together with Lemma \ref{lem:lower_semicontiuous}, in turn implies that 
		\begin{align}\label{eq:GammaForFullyDiscrete}		
		P_\gamma(u) &\leq \liminf\nolimits_{k\to \infty} Q_\gamma^k(u^k)  \\
		&\leq \liminf\nolimits_{k\to \infty} P_\gamma^k(u^k)
		+ \limsup\nolimits_{k\to \infty} \sup\nolimits_{n \geq k} |P_\gamma^n(u^n) - Q_\gamma^n(u^n)| . \notag
	  \end{align}
   Since the convergence in \eqref{eq:diffPandQ} is uniform in $u$ the second summand on the right hand side is $0$
   which shows \eqref{eq:gamma_conv_1}. In order to obtain a recovery sequence for the discrete Potts functionals $(P_\gamma^k),$
   we consider a recovery sequence $(u^k)$ for the semidiscrete Potts functionals $(Q_\gamma^k),$ which exists according to
   Lemma \ref{lem:recovery_sequence}. Then, also by Lemma \ref{lem:recovery_sequence},
   		\begin{align*}
		P_\gamma(u) &\geq \limsup\nolimits_{k\to \infty} Q_\gamma^k(u^k) \\ & \geq \limsup\nolimits_{k\to \infty} P_\gamma^k(u^k)
		- \limsup\nolimits_{k\to \infty} \sup\nolimits_{n \geq k} |P_\gamma^n(u^n) - Q_\gamma^n(u^n)| .
		\end{align*}
    As above, the second summand on the right hand side is $0$ and so $(u^k)$ is a recovery sequence for $(P_\gamma^k),$ too.
    This shows \eqref{eq:gamma_conv_2} and we have shown $\Gamma$-convergence. 
   
    Our next goal is to locate all minimizers of $(P_\gamma^k)$ in a common compact set $M$. To this end, we show 
    the analogous statement to \eqref{eq:strong_equi_mildly} replacing the semidiscrete Potts functional by
    their discrete counterparts, i.e.,     
    \begin{align} \label{eq:strong_equi_mildly_discr}
    \inf_{u \in M} P_\gamma^k(u) = \inf_{u \in \PC} P_\gamma^k(u) < P_\gamma^k(v)  \text{ for all } v \notin M.
    \end{align}
    The proof is based on the corresponding proof of \eqref{eq:strong_equi_mildly} which can be found in
    the proof of \autoref{thm:gamma_convergence_semi}.
    
    {\noop We first consider the case where $S_k$ is the integral sampling operator \eqref{eq:sampling_integral}.
    We choose the compact set $M=M_{C,j}$ with  $C$ and $j$ as in the proof of \autoref{thm:gamma_convergence_semi}.
    In particular, this yields the implication  \eqref{eq:boundOnIntervallengthImpliesSmallIntegralA} 
    for $g=f^k$ and $\varepsilon = \gamma/6.$ Now we can consider a piecewise 
    constant function $u \notin M$ whose jump set is contained in $X^k$ and 
    proceed as in the proof of \autoref{thm:gamma_convergence_semi} to obtain a piecewise constant function $\tilde{u}$    
    whose jump set is contained in $X^k,$ which has smaller $P_\gamma^k$ value, 
    and which is contained in $M_{C,j}.$ This in turn implies \eqref{eq:strong_equi_mildly_discr}.}

    {\noop If $S_k$ is the point sampling operator \eqref{eq:sampling_point} and if $f$ is continuous, we have  
    $\|f^k\|_1$ $\leq \|f^k\|_\infty$ $\leq \|f\|_\infty.$ Then \eqref{eq:boundOnIntervallengthA} applied to $f^k$ yields
    the implication \eqref{eq:boundOnIntervallengthDiscrA} with $\|f\|_1$ replaced by $\|f\|_\infty.$
    In this case \eqref{eq:boundOnIntegralDiscrA} is trivial since $\int\nolimits_I |f^k|$ $\leq \|f\|_\infty \cdot |I|.$ 
    As in the case of integral sampling we use these two implications to establish \eqref{eq:boundOnIntervallengthImpliesSmallIntegralA} for $g=f^k$
    which allows us to proceed as in the proof of \autoref{thm:gamma_convergence_semi} to obtain \eqref{eq:strong_equi_mildly_discr}.}
    
    {\noop Each $P_\gamma^k$ has a minimizer since it is equal to a semi-discrete Potts functional for data $\sum\nolimits_j S_k f(j) \indfunc_{I_j}$
    which has a minimizer by \autoref{thm:gamma_convergence_semi}.}
    Equipped with \eqref{eq:strong_equi_mildly_discr} and knowing that each $P_\gamma^k$ has a minimizer 
    we conclude that all minimizers are contained in $M = M_{C,j}.$
    The compactness of $M$ implies the existence of a cluster point for any  sequence of minimizers $(u^k)$.
    Then the last assertion of the theorem is a consequence of the fundamental theorem of $\Gamma$-convergence;   
    cf. \autoref{thm:fundGammaCon} or \cite[Theorem 1.21, p. 29]{braides2002gamma}.  \quad
\endproof

\section{A fast algorithm for the exact minimization of $L^1$-Potts functionals}\label{sec:alogrithm}

In this section we propose a fast algorithm to minimize weighted Potts functionals of the form 
\begin{align}\label{eq:potts_discrete_non_equidistant} 
\| u - f\|_{\ell^1_w} + \gamma J(u) = \sum\nolimits_{i} w_i |u_i - f_i| + \gamma \cdot \numel{\{i : u_i \neq u_{i+1} \}},
\end{align}
where the $w_i$ are positive weights.
Weighted $\ell^1$-data terms appear naturally when data are not sampled equidistantly.
Clearly, if all weights $w_i$ are identical
this reduces to the non-weighted $L^1$-Potts functional \eqref{eq:potts_L1_discFir}.

In \autoref{sec:PottsCore} we recall the Potts core algorithm as proposed in \cite{friedrich2008complexity}.
In \autoref{sec:Histogram} we introduce a suitable dynamic data structure.
In \autoref{sec:OurAlgorithm} we explain our fast algorithm and obtain results on its complexity in \autoref{ssec:complexity}.
In \autoref{sec:TheAlgorithmPlainExperiments} we experimentally compare $L^1$-Potts, $L^2$-Potts, and  $L^1$-TV minimization under various types of noise.

\subsection{The Potts core algorithm}\label{sec:PottsCore}

{ \noop
The basic idea is that a minimizer of the Potts functional
for data $(f_1,..., f_r)$ can be computed in polynomial time provided that minimizers of the  reduced data $f_{[1,1]},$ $f_{[1,2]},$ $...,$ $f_{[1,r-1]}$ are known
\cite{mumford1985boundary, chambolle1995image}.
Here, we use the notation $f_{[\ell, r]}$ for the partial data $(f_\ell, f_{\ell+1}, ..., f_{r}).$ 
The procedure works as follows.
Let us denote  by $u^1,$  $u^2,$ ..., $u^{r-1}$ arbitrary (not necessarily unique) minimizers
of the Potts functional for the partial data $f_{[1,1]},$ $f_{[1,2]},$ $...,$ $f_{[1,r-1]},$ respectively.
In order to compute a minimizer for data $f_{[1,r]},$ 
we first create a set of $r$ candidates $v^1,$ ..., $v^{r},$ each of length $r,$
which are given by
\begin{equation}\label{eq:potts_candidate}
	v^\ell = (u^{\ell-1} , \underbrace{\mu_{[\ell,r]},..., \mu_{[\ell,r]}}_{\text{Length } r - \ell +1 }). 
\end{equation}
Here, $u^0$ is the empty vector and
the symbol $\mu_{[\ell, r]}$ denotes an (arbitrary) weighted median for data $f_{[\ell, r]}.$ 
Recall that a weighted median is a minimizer of the functional 
$y \mapsto  \| (y, ..., y) - f_{[\ell, r]}\|_{\ell^1_w}$
and that the minimal functional value is given by the median deviation
\begin{align}\label{eq:med_dev}
	d_{[\ell, r]} = \| f_{[\ell, r]} - (\mu_{[\ell, r]} , ..., \mu_{[\ell, r]})  \|_{\ell^1_w}
	 = \sum\nolimits_{i=\ell}^{r} w_i | f_{i} - \mu_{[\ell, r]} |.
\end{align}
Among the candidates $v^\ell,$ there is a minimizer of the Potts functional, as the following lemma shows.
\begin{lemma}\label{lem:potts_minimizer_cand}
For the Potts functional associated with data $f_{[1,r]},$ we have that
\[
\min_{\ell=1,...,r} P_{\gamma}(v^\ell) = \min_{u \in \R^r} P_{\gamma}(u).
\]
In particular, the Potts functional value of a minimizer $u^*$ is given  by 
\begin{align}\label{eq:potts_value_candidate}
P_\gamma(u^*) = \min_{\ell=1,...,r} P_{\gamma}(u^{\ell-1}) +  \gamma + 	d_{[\ell,r]},
\end{align}
where  we set $P_\gamma(u^0) = -\gamma.$
\end{lemma}
\begin{proof}
Let $u^*$ be a minimizer of the Potts functional 
for data $f_{[1,r]}.$ 
If $u^*$ is constant then  $P_{\gamma}(u^*) = P_{\gamma}(v^1)$ because both $u^*$ and $v^1$ are constant vectors containing a median of data
 $f_{[1, r]};$ hence $v^1$ is a minimizer of the Potts functional. 
In particular,
we have $P_\gamma(u^*) = d_{[1, r]} = P_\gamma(u^0) + \gamma + d_{[1,r]}$
which shows \eqref{eq:potts_value_candidate} in the constant case.

If $u^*$ is not constant, let $\ell^* \in \{1,...,r-1\}$ be its rightmost jump location, i.e.,
$u^*_{\ell^*} \neq u^*_{\ell^*+1}$ and $u^*$ is constant on $\ell^*+1,...,r.$
Since $u^{\ell^*}$ minimizes the Potts functional associated with data $f_{[1,\ell^*]}$ and since $u^*$ has a jump at $\ell^*$ we have 
for the candidate $v^{\ell^*+1},$ which jumps at $\ell^*,$ that
\begin{align}\label{eq:potts_candidate_value_estimate}
	P_\gamma(v^{\ell^*+1}) &\leq
	 P_\gamma(u^{\ell^*}) + \gamma + d_{[\ell^*+1, r]}
	\leq P_\gamma(u^*_{[1,\ell^*]}) + \gamma + 
	d_{[\ell^*+1, r]}
	= P_\gamma(u^*).
\end{align}
Here $u^*_{[1,\ell^*]}$ denotes the vector of the first $\ell^*$ components of $u^*.$
Hence, $v^{\ell^*+1}$ is a minimizer of the Potts functional.
Further, since $P_\gamma(v^\ell) \leq P_\gamma(u^{\ell-1}) + \gamma + d_{[\ell, r]}$ for all $\ell= 1,...,r$
we have that
\begin{align*}
\min_{\ell=1,...,r} P_\gamma(v^\ell)
\leq\min_{\ell=1,...,r} P_\gamma(u^{\ell-1}) + \gamma + d_{[\ell, r]}
\leq
P_\gamma(u^{\ell^*}) + \gamma + d_{[\ell^*+1, r]}
\leq P_\gamma(u^*).
\end{align*}
Using $P_\gamma(u^*) \leq \min_{\ell=1,...,r} P_\gamma(v^\ell)$ we obtain \eqref{eq:potts_value_candidate} in the non-constant case. \quad
\end{proof}

Recall from Example \ref{ex:nonuniqueness} that minimizers of the $L^1$-Potts functional are not unique in general;
so it is possible that several candidates $v^{\ell}$ fulfill $P_\gamma(v^{\ell}) =  P_\gamma(u^*).$
Note that, for such a minimizer $v^{\ell}$ ,  the values $v^{\ell}_{\ell-1}$ and $v^{\ell}_{\ell}$ 
are different, i.e., $v^{\ell}$ in fact has a jump at $\ell -1$
(despite the possible non-uniqueness of medians). 
This is a consequence of the proof of Lemma~\ref{lem:potts_minimizer_cand} above since, otherwise, we could replace the first ``$\leq$'' symbol in   
\eqref{eq:potts_candidate_value_estimate} by ``$<$'' which would contradict $u^\ast$ being a minimizer.


Formula \eqref{eq:potts_value_candidate} already yields a procedure
to minimize the Potts functional.
A speedup is given by a crucial observation of \cite{friedrich2008complexity}:
we can compute a minimizer from knowing the rightmost jump locations of (arbitrary) 
minimizers of certain data.
 

\begin{lemma}\label{lem:reconstruction_of_minimizers}
For $\ell= 1,\ldots,r,$ let $Z(\ell)$ be the rightmost jump location of an (arbitrary) minimizer 
$u^\ell$ of the Potts functional associated with data $f_{[1,\ell]}.$ ($Z(\ell) = 0$ if $u^\ell$ is constant.)
We consider the strictly decreasing finite sequence $r, Z(r),$ $Z(Z(r)),$ $Z(Z(Z(r))),...,0.$
We define $\bar{u}$ on $Z^{k+1}(r) + 1,..., Z^k(r)$ as a median $\mu_{[Z^{k+1}(r) + 1, Z^k(r)]}.$
Then $\bar{u}$ is a minimizer of the Potts functional for data $f_{[1,r]}.$
\end{lemma}
\begin{proof}
We use induction on $r.$ For $r=1,$ $Z(r)=0$ and $u=f_1=\mu_{[1,1]}$ is a minimizer of $P_\gamma$ for data $f_1.$
As induction hypothesis, we assume that the statement is true for all $\ell = 1,\ldots,r-1.$
We consider data $f_{[1,r]}$ and the minimizer $u^r$ whose rightmost jump location is $Z(r)<r.$
If $Z(r)=0,$ then $u^r$ is constant. By the minimizing property of a median, 
$\|\bar{u}-f_{[1,r]}\|_{\ell_w^1}$ $\leq \|u^r-f_{[1,r]}\|_{\ell_w^1}$ and thus $P_\gamma(\bar{u}) \leq P_\gamma(u^r).$ Hence,
$\bar{u}$ is a minimizer.
If $Z(r) \neq 0,$ we denote the first $Z(r)$ components of $\bar{u}$ by $\bar{u}'.$ $\bar{u}'$ is a minimizer of 
$P_\gamma$ for data $f_{[1,Z(r)]}$ by the induction hypotheses. Thus, 
$$
 P_\gamma (\bar{u}) \leq P_\gamma (\bar{u}') + \gamma + d_{[Z(r)+1,r]} \leq P_\gamma (u^r|_{\{1,\ldots,Z(r)\}}) + \gamma + d_{[Z(r)+1,r]} \leq P_\gamma (u^r).
$$
This implies that $\bar{u}$ minimizes $P_\gamma$ for data $f_{[1,r]}$ which completes the induction.
\end{proof}

Using Lemma~\ref{lem:reconstruction_of_minimizers} together with 
\eqref{eq:potts_value_candidate} we get the following {\em Potts core algorithm}
which is the core part of Algorithm~\ref{alg:findBestPartitionL1}.
Assume that, for $\ell= 1,\ldots,r-1,$ $Z(\ell)$ and $P_\gamma(u_\ell)$ 
have been calculated and stored. 
Calculate $P_\gamma(u^r)$
and $Z(r)$ of an arbitrary minimizer $u^r$ by \eqref{eq:potts_value_candidate}. 
We store these values and proceed with $r+1$ until we reach $n.$
Then we use Lemma~\ref{lem:reconstruction_of_minimizers} to construct a minimizer
for full data $f_{[1,n]},$ from the rightmost jump locations.
To get a deterministic algorithm we use the following selection rules: 
we take a candidate in \eqref{eq:potts_value_candidate} 
with smallest $l$, and we always choose smallest medians (when not unique).
The Potts core algorithm 
needs at least $O(n^2)$ steps, 
since we iterate over $r = 1,...,n$ in an outer loop and  over $\ell = 1,..., r$ in an inner loop.
Optimal time complexity can only be achieved if the operations in the inner  $\ell$-loop take constant time.
The critical computation in the inner loop is the evaluation of
the righthand side of \eqref{eq:potts_value_candidate} for fixed $\ell,$ $\ell=1,...,r$. 
The stored Potts functional value $P_\gamma(u^{\ell-1})$ can be simply
looked up in constant time. 
Thus, the remaining time critical computations
are the evaluations of the median deviations $d_{[\ell,r]}.$
So the challenge to get a fast $L^1$ Potts algorithm
is to provide fast computation of the $O(n^2)$ median deviations $d_{[\ell,r]}$ 
within $O(n^2)$ time and to perform calculations \emph{in situ} in order not to consume extra memory.
}
\subsection{The indexed linked histogram}\label{sec:Histogram} 
For the fast computation of the median deviations \eqref{eq:med_dev}
we introduce the dynamic data structure \emph{indexed linked histogram.} 
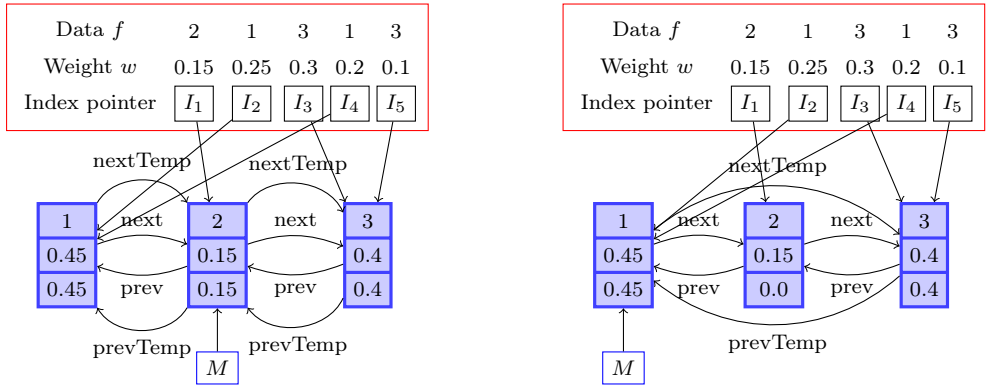
\begin{figure}
\centering
  \tikzstyle{histnode}=[rectangle split, rectangle split parts=3,thick,draw=blue!75,fill=blue!20, very thick]
  \tikzstyle{data}=[rectangle,draw=red!75,fill=red!20]
  \tikzstyle{point}=[draw=black,minimum size=1mm]
  \tikzstyle{median}=[rectangle,draw=blue,minimum size=3mm]
\begin{tikzpicture}[node distance=2.cm,bend angle=45,auto, scale=0.7, font=\footnotesize]
  \tikzstyle{every label}=[red]
\matrix [draw=red] (dataorg)
{
\node {Data $f$}; &
\node {2};  &
\node {1}; &
\node {3}; &
\node {1}; &
\node {3}; \\ 
\node {Weight $w$}; &
\node {0.15};  &
\node {0.25}; &
\node {0.3}; &
\node {0.2}; & 
\node {0.1}; \\
\node  {Index pointer}; &
\node [point] (p1) {$I_1$}; &
\node [point] (p2) {$I_2$}; &
\node [point] (p3) {$I_3$}; &
\node [point] (p4) {$I_4$}; &
\node [point] (p5) {$I_5$}; \\
};
       \node [histnode] (d2) [below of=dataorg, node distance=2.5cm]                      {2 \nodepart{second} 0.15 \nodepart{third} 0.15	};
    \node [histnode] (d1) [left of=d2]                                      {1 \nodepart{second} 0.45 \nodepart{third} 0.45};
    \node [histnode] (d3) [right of=d2]                      {3 \nodepart{second} 0.4 \nodepart{third} 0.4};
      \path[->] (d1) edge[bend left=20] node [above] {next} (d2) (d2) edge[bend left=20] node [above] {next} (d3);
      \path[->] (d3) edge[bend left=20] node [below] {prev} (d2) (d2) edge[bend left=20] node [below] {prev} (d1);
      \path[->] (d1) edge[bend left=60] node [above] {nextTemp} (d2) (d2) edge[bend left=60] node [above] {nextTemp} (d3);
      \path[->] (d3) edge[bend left=60] node [below] {prevTemp} (d2) (d2) edge[bend left=60] node [below] {prevTemp} (d1);
	\path[->] (p1) edge (d2);
	\path[->] (p2) edge (d1);
	\path[->] (p3) edge (d3);
	\path[->] (p4) edge (d1);
	\path[->] (p5) edge (d3);
	\node [median] (med) [below of=d2, yshift=0.5cm]                      {$M$};
	\path[->] (med) edge (d2);
\end{tikzpicture}
\hfill
\begin{tikzpicture}[node distance=2.cm,bend angle=45,auto, scale=0.7, font=\footnotesize]
\matrix [draw=red] (dataorg)
{
\node {Data $f$}; &
\node {2};  &
\node {1}; &
\node {3}; &
\node {1}; &
\node {3}; \\ 
\node {Weight $w$}; &
\node {0.15};  &
\node {0.25}; &
\node {0.3}; &
\node {0.2}; & 
\node {0.1}; \\
\node  {Index pointer}; &
\node [point] (p1) {$I_1$}; &
\node [point] (p2) {$I_2$}; &
\node [point] (p3) {$I_3$}; &
\node [point] (p4) {$I_4$}; &
\node [point] (p5) {$I_5$}; \\
};
     \node [histnode] (d2) [below of=dataorg, node distance=2.5cm]                      {2 \nodepart{second} 0.15 \nodepart{third} 0.0	};
    \node [histnode] (d1) [left of=d2]                                      {1 \nodepart{second} 0.45 \nodepart{third} 0.45};
    \node [histnode] (d3) [right of=d2]                      {3 \nodepart{second} 0.4 \nodepart{third} 0.4};
      \path[->] (d1) edge[bend left=20] node [above] {next} (d2) (d2) edge[bend left=20] node [above] {next} (d3);
      \path[->] (d3) edge[bend left=20] node [below] {prev} (d2) (d2) edge[bend left=20] node [below] {prev} (d1);
      \path[->] (d1) edge[bend left=40] node [above] {nextTemp}  (d3);
      \path[->] (d3) edge[bend left=40] node [below] {prevTemp}  (d1);
	\path[->] (p1) edge (d2);
	\path[->] (p2) edge (d1);
	\path[->] (p3) edge (d3);
	\path[->] (p4) edge (d1);
	\path[->] (p5) edge (d3);
	
		\node [median] (med) [below of=d1, yshift=0.5cm]                      {$M$	};
	\path[->] (med) edge (d1);
\end{tikzpicture}
\caption{Data structure indexed linked histogram.
The shaded rectangles depict the histogram nodes with the entries \emph{value,} \emph{weight,}  and \emph{temporary weight} (from top to bottom).
 Left: The temporary pointers at the beginning of the (inner) $\ell$-loop coincide with the non-temporary ones. Right: In the first $\ell$-iteration, the left-most weight $w_1$ is subtracted from the temporary weight of the node corresponding to $f_1 = 2.$
 Since the temporary weight of that node then equals zero, 
 it is removed from the temporary list. 
The temporary median pointer $M$ is shifted accordingly and the deviation is updated.}
\label{fig:indexedLinkedList}
\figspace
\end{figure}

\paragraph{Basic structure} 

A histogram associated with data $f = (f_1,\ldots,f_m)$ and weights $w_i$ is the image measure $f(w)$ of the measure $w$ given by the $w_i$ under the mapping $f.$ Thus the domain of the histogram is equal to the set of data values of $f$ and it assigns to each data value $f_j$
the sum $\sum_{i: f_i=f_j } w_i$ of the weights $w_i$ of all data points with value $f_j.$
An indexed linked histogram basically stores a histogram; for an example see \autoref{fig:indexedLinkedList}. It has four crucial characteristics.
(i) The histogram is realized as linked list
where the pairings \enquote{value $\mapsto$ weight} are realized as nodes called \emph{histogram nodes.}
(ii) The linked list is sorted by the values of the histogram nodes.
(iii) The histogram nodes have a twofold structure, that is,
they have permanent pointers and weights, and extra temporary pointers and temporary weights.
This allows to store a temporary histogram for partial data $(f_{[\ell,r]})$ without losing the histogram for the full data.
(iv) An array $I$ of \emph{index pointers}
links the position $i$ of each data item $f_i$ with its corresponding histogram node.

\paragraph{Building the histogram}\label{ssec:linked_hist}
Assume that we have a histogram for the data $f_{[1, r]}.$ Then we obtain a histogram for the data $f_{[1, r+1]}$ by
adding the weight $w_{r+1}$ to the histogram node $N$ with value $f_{r+1}$ (which has to be created, if not existent.) 
This involves searching (and possibly inserting) in a linked list which is a linear time task. 
Furthermore, the index pointer $I_{r+1}$ is set to point to $N.$
Then we find the median histogram node and store it in the \emph{median pointer} $M.$
The median pointer can be computed from scratch by a simple iteration over the sorted histogram in linear time
or, more efficiently, by restoring and updating the median pointer of the $r$th iteration.
We compute the deviation $d_{[1, r]}$ of the median from the data $f_{[1, r]}$ by simple summation. 
At last, we initialize the temporary structure by the permanent structure.
These are all linear time tasks.

\paragraph{Fast computation of median deviations}
Assume that, for data $f_{[\ell, r]},$ we are given an indexed linked histogram, 
a pointer to its median node, and the median deviation $d_{[\ell,r]}.$ Then we compute the median deviation $d_{[\ell+1, r]}$ for the reduced data $f_{[\ell+1, r]}$ as follows. 
\begin{enumerate}
\item[(i)] We remove the weight $w_\ell$ from the histogram node corresponding to the data $f_\ell.$ The crucial point is that
due to the indexing by the array $I$ we find this histogram node in constant time which is the time-critical task.
\item[(ii)] The weights below or above the (old) median node may have more than half of the total weight.
Therefore, we shift the median pointer downwards or upwards until we reach the (new) median node of the reduced histogram. 
Hereby, the number of required shifting operations depends on the fraction between maximum and minimum weight
as follows.
\begin{align}\label{eq:PointerShifts}
 \text{Updating the median requires  at most } 
 \left\lceil \frac {\max_i w_i}{\min_i w_i} \right\rceil
 \text{ pointer shifts. }
\end{align}
\item[(iii)] Along with each shift operation we update the median deviation, which is of constant runtime for each shift operation. Hence,
\begin{align} \label{eq:StepShifts} 
 \text{Computing the deviation $d_{[\ell+1,r]}$ needs at most } 
 \left\lceil \frac {\max_i w_i}{\min_i w_i} \right\rceil
 \text{ steps. }
\end{align} 
\end{enumerate}

\subsection{The fast $L^1$-Potts algorithm}\label{sec:OurAlgorithm}

We connect our data structure with the Potts core algorithm.
Our $L^1$-Potts algorithm is outlined in \autoref{alg:findBestPartitionL1}.

Using the indexed linked histogram we compute
the deviations $d_{[\ell,r-1]}$ as follows.
First, starting from a histogram for data $f_{[1,r-1]}$ we add data $f_{r}$
as explained in the paragraph on \emph{building the histogram} in \autoref{ssec:linked_hist}.
We obtain the deviation $d_{[1, r]}.$ 
To obtain the deviations $d_{\ell, r}$ we successively 
remove the left hand data items $f_\ell$ as described in the paragraph on \emph{fast computation of median deviations} in \autoref{ssec:linked_hist}.
This removal is done on the temporary structure of the histogram in order to save the histogram for data $f_{[1, r]}$ for the next update $r \to r+1.$

Moreover our data structure allows for \emph{in situ} computation of the $d_{[\ell,r]}$
right at the place where they are needed such that no additional memory is required.
More precisely, each deviation $d_{[1,r]},$ $d_{[2,r]},$ ..., $d_{[r,r]}$ produced in the inner $\ell$-loop
can immediately be used to compute  the values
	$P_{\gamma}(u^{\ell-1}) +  \gamma + 	\| f_{[\ell, r]} - (\mu_{[\ell, r]} , ..., \mu_{[\ell, r]})  \|_{\ell_w^1},$  ($\ell = 1,..., r$), cf. \eqref{eq:potts_value_candidate}.
After passing to the infimum over $\ell$ we obtain a minimizer for the data $f_{[1, r]}$
according to \eqref{eq:potts_value_candidate}, and the above deviations are not required any more.

For further details and implementational aspects we refer to our Matlab implementation available for download at \url{http://pottslab.de}.

\begin{algorithm2e}
\footnotesize
\caption{Minimization of the $L^1$-Potts functional}
\label{alg:findBestPartitionL1}
\SetKwInOut{Input}{input}
\SetKwInOut{Output}{output}
\SetKwInOut{Local}{local}
\SetCommentSty{text}
\SetCommentSty{footnotesize}
\Input{Data vector $f \in \R^n, $ weight vector $w \in \R_+^n.$}
\Output{Minimizer $u$ of the (weighted) $L^1$-Potts functional}
\Local{Left and right interval bounds $\ell,r \in \N$;
 Potts values $P \in \R^{n+1}$;
 temporary values $p, \mu, d \in\R$ (candidate Potts value, median, and median deviation); indexed linked histogram $H$;
 array of right-most jumps $Z \in \N^n$; pointer to histogram node $M$; } 
\Begin{
$P_0 \leftarrow -\gamma$\;
$H \leftarrow$ Empty indexed linked histogram\; 
\For{$r \leftarrow 1$ \KwTo $n$}{
	$P_r \leftarrow \infty$ \tcc*{init candidate Potts value}
	Insert element $f_r$ with weight $w_r$ to histogram	$H$\;
		Set temporary values and pointers of the histogram $H$ to the permanent ones\;
	$M \leftarrow$ Pointer to median node of histogram $H$\;
	$\mu \leftarrow $ value of median node $M$ \tcc*{init the current median $\mu$}
	$d \leftarrow \sum_{i=1}^r w_i | \mu - f_i | $ \tcc*{init the current median deviation $d$}
	\For{$\ell \leftarrow 1$ \KwTo $r$}{
		$p \leftarrow P_{\ell-1} + \gamma + d$ \tcc*{compute candidate Potts value $p$}
		\tcc{if $p$ is lower than previous best Potts value $P_r$}
		\If{$p \leq P_r$}{
		$P_r \leftarrow p$ 	\tcc*{store new best Potts value}
			$Z_r \leftarrow \ell - 1$ \tcc*{update right-most jump}	
	}
	Temporarily remove the weight $w_\ell$ from the node which
	the  $I_\ell$ (cf. \autoref{sec:Histogram}) points to (if the weight of that node then equals zero, remove it from $H$)\;	
	Shift the median pointer $M$ towards the new median node of $H$
	and update the current median deviation $d$\;
}
}
Reconstruct the minimizer $u$ from the array of right-most jumps $Z$\;
}
\end{algorithm2e}

\subsection{Theorem on time and space complexity}\label{ssec:complexity}

We consider the continuous Potts functional $P_\gamma$ defined by \eqref{eq:potts_continuous_L1} 
and its discretizations $(P^k_\gamma)$ defined by \eqref{eq:potts_fully_discrete}.
The piecewise constant functions associated with $P^k_\gamma$ only jump in the finite sets $X^k$
(cf. \autoref{sec:gamma-convergence}) which allows us to identify such a piecewise constant function
with a vector in euclidean space. Letting $w^{(k)}_i =\lambda(I_j^{(k)})$ be the length of the interval
$I_j$ (cf. \eqref{eq:potts_fully_discrete}) the discretized Potts problem \eqref{eq:potts_fully_discrete}
reads 
\begin{align}
\gamma J(u) +  \sum\nolimits_{i} w_i |u_i - (S_k f)_i|  
 = \gamma J(u) + \| u - S_k f \|_{\ell^1_w} \to \min
\end{align}
which is a weighted discrete $L^1$-Potts problem of type \eqref{eq:potts_discrete_non_equidistant}. 
If the weights behave well, we have the following statement on the complexity of the algorithm
if the mesh size goes to $0$.
\begin{theorem}\label{thm:timeSpaceComplNonAqui}
  Consider a sequence of weight vectors $w^{(1)},w^{(2)},\ldots$ where each weight vector 
  $w^{(k)} = (w^{(k)}_1,\ldots,w^{(k)}_{n_k})$ is of length $n_k$ and the sequence $n_k$ increases.
  We assume that 
  \begin{align}\label{eq:AssumptionOnWi}
    \frac{\max_i w^{(k)}_i} {\min_i w^{(k)}_i } \leq C  \qquad \text{with $C$ independent of $k$.}
  \end{align}  
  Then the proposed $L^1$-Potts algorithm computes an exact minimizer of the 
  discrete weighted $L^1$-Potts functional within 
  $O(n_k^2)$ time and $O(n_k)$ space as $k$ increases. 
\end{theorem}

In the case of equal weights we obtain the following corollary.
\begin{corollary}
  The proposed $L^1$-Potts algorithm computes an exact minimizer of the 
  discrete $L^1$-Potts functional \eqref{eq:potts_L1_discFir} within $O(n^2)$ time and $O(n)$ space.
\end{corollary}

{\em Proof of \autoref{thm:timeSpaceComplNonAqui}.}
We begin with the time complexity. In the paragraph on the Potts core algorithm we have seen 
that the Potts core algorithm needs $O(n^2)$ operations. Assumption \eqref{eq:AssumptionOnWi} ensures that 
the medians $\mu_{[l,r]}$ in the inner $\ell$-loop of \autoref{alg:findBestPartitionL1} can be updated within $O(1);$ 
cf. \eqref{eq:StepShifts}. 
The same is true for the deviations $d_{[\ell,r]}$ given by \eqref{eq:med_dev}; cf. \eqref{eq:PointerShifts}.
This guarantees that the operations in the inner $\ell$-loop are of constant complexity $O(1).$
All operations in the outer $r$-loop are of complexity $O(n)$ as explained in the paragraph
\emph{building the histogram} in \autoref{ssec:linked_hist}.
The final reconstruction step is of order $O(n \log n).$ 
(This can be further reduced to $O(n),$ if the medians computed in the main iteration are stored.)
Hence, the overall time complexity is $O(n^2).$

 We come to space complexity. 
The only items depending on the length of input data 
 in the data structure indexed linked histogram are the number of histogram nodes 
 and the indexing array $I.$
  Both have at most length $n.$
  Furthermore the arrays $P$ and $Z$ storing the Potts values $P_\gamma(u^r)$
  and the right-most jump locations (see \autoref{sec:PottsCore})
  have length of order $n.$
  Summing up, \autoref{alg:findBestPartitionL1} has space complexity $O(n).$

{\noop 
It remains to show the correctness of the proposed algorithm. For all partial data $f_{[1,r]}$, the algorithm computes a rightmost jump location $Z(r)= \ell^* -1$ 
where $\ell^*$ is a minimizing argument in \eqref{eq:potts_value_candidate}. It follows from Lemma \ref{lem:potts_minimizer_cand} that this $Z(r)$ is the jump location of a minimizer  
for data $f_{[1,r]}.$ From these data $Z(r),$ $r=1,\ldots,n,$ it produces a function which according to Lemma \ref{lem:reconstruction_of_minimizers} is a minimizer    
of the $L^1$ Potts problem.}
%
\endproof

We give some comments.
We note that our data structure allows in situ computation of the $d_{[\ell, r]}$
which makes our algorithm an $O(n)$ space algorithm.
The asymptotic time and space complexity of our $L^1$-Potts algorithm equals 
that of the $L^2$-Potts algorithm proposed by F. Friedrich et al.\ \cite{friedrich2008complexity}.
Remarkably, our algorithm is not only asymptotically of comparable runtime.
The actual runtime of the $L^1$-Potts algorithm  is only about $20 \%$ higher than that of the $L^2$-Potts algorithm, for reasonably large signals; we refer to 
 \autoref{fig:potts_runtime} for a comparison of runtimes.
 
\begin{figure}
\centering
\pgfkeys{/pgfplots/axis labels at tip/.style={
        xlabel style={
            yshift=-1ex
        },
        ylabel style={
            yshift=-1ex
        }
    }
}
	\begin{tikzpicture}
      \begin{loglogaxis}[
      compat=newest,
      ylabel shift={-1ex},
      xlabel shift={-1ex},
        font=\scriptsize,
        width=0.53\textwidth,
        height=0.35\textwidth,
        xlabel=$n$,
        ylabel= Runtime in seconds,
        xlabel= Length of signal $n$,
        legend style={at={(0.05,0.95)},anchor=north west},
        grid=major
        ]
        \pgfplotsset{every axis plot post/.append style={line width = 0.5pt}} ;
        \addplot table[x index=0,y =RuntimeL1W] {Experiments/runtimes/runtimes.table};   
         \addlegendentry{$L^1$-Potts};
        \addplot table[x index=0,y =RuntimeL2] {Experiments/runtimes/runtimes.table};   
        \addlegendentry{$L^2$-Potts};
         \addplot[mark=none, dashed, domain=1000:20000] {(x/8000) ^ 2 };  %
        \addlegendentry{$O(n^2)$};
      \end{loglogaxis}
    \end{tikzpicture}
    \hfill
    	\begin{tikzpicture}
      \begin{semilogxaxis}[
        font=\scriptsize,
              compat=newest,
        ylabel shift={-1ex},
      xlabel shift={-1ex},
        width=0.53\textwidth,
        height=0.35\textwidth,
        scaled ticks=false,
        ymin=0,
        ytick={0,1,2},
        ylabel= $\frac{ \text{Runtime $L^1$-Potts} }{ \text{Runtime $L^2$-Potts}}$,
        xlabel= Length of signal $n$,
        tick label style={font=\footnotesize},
        grid=major
        ]
        \pgfplotsset{every axis plot post/.append style={line width = 0.5pt}} ;
        \addplot table[x index=0,y =divisionW] {Experiments/runtimes/runtimes.table};   
     \end{semilogxaxis}
    \end{tikzpicture}
    \caption{Runtimes of the $L^1$ and $L^2$-Potts algorithms on a logarithmic scale  (left) and relative runtimes (right).
    For larger signals, the $L^1$ algorithm is almost as fast as the $L^2$ algorithm.
    The slower runtime for smaller signals is mainly due to overhead in the histogram structure. 
    The experiments were conducted on an Apple MacBook Pro, with Intel Core 2 Duo 2.66 GHz and 8 GB RAM,
using our implementation (\url{http://pottslab.de}).
     }
     \label{fig:potts_runtime}
     \figspace
\end{figure}
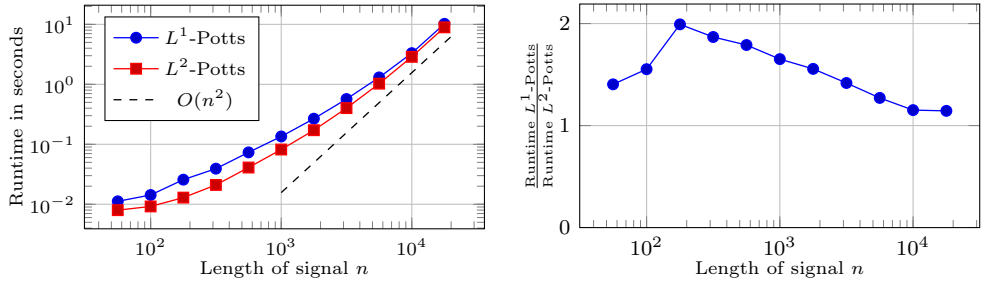

We next show that the condition \eqref{eq:AssumptionOnWi} in \autoref{thm:timeSpaceComplNonAqui} cannot be dropped. To this end, we consider data $f_i$ and weights $w_i$ given by the following table,
\begin{align*}
\begin{array}[c]{c|c|c|c|c|c|c}
i   & 1& 2& 3 & 4 &   \ldots & n    \\ \hline
f_i & 1& n& 2 & n-1 &   \ldots& n/2  \\ \hline
w_i & 1/2 & 1/4 & 1/8 & 1/16   & \ldots & 1/ 2^{n+1}     \\
\end{array}
\end{align*}
We see that, for updating the median in the inner $l$-loop, our algorithm shifts the median pointer from the outer left side to 
the outer right side and vice versa, respectively. Thus the complexity of the median updating is of order $n,$ and the algorithm 
is no longer of complexity $O(n^2).$

\subsection{Numerical experiments on the $L^1$-Potts functional} \label{sec:TheAlgorithmPlainExperiments}

\def\thisfigheight{0.35\textwidth}
\def\thisfigwidth{0.34\textwidth}
\setlength{\tabcolsep}{1pt}

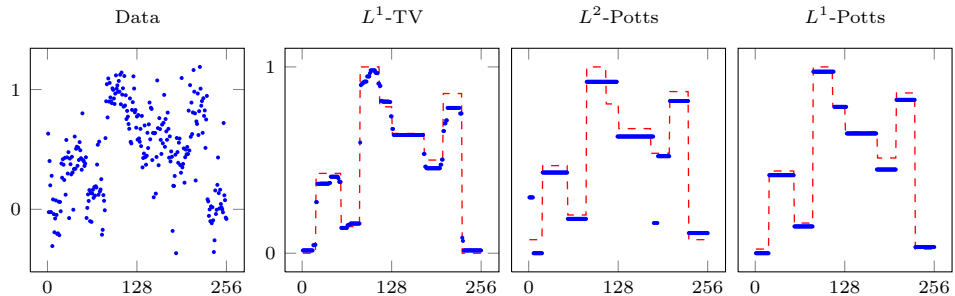
\begin{figure}
\centering
\def\noiseLevel{0.2}
\def\noiseType{LaplaceNoise}
\def\figfolder{ExperimentsRev}
\begin{tabular}{cccc}
   	\begin{tikzpicture}
      \begin{axis}[
        MyAxisStyle,
        title={Data}
        ]
        \pgfplotsset{every axis plot post/.append style={line width = 0.5pt}} ;
        \addplot+[only marks, mark size=0.5] table[x index=0,y index=2] {\figfolder/\noiseType2/\noiseType\noiseLevel.table};   
     \end{axis}
    \end{tikzpicture}
    &
     \begin{tikzpicture}
      \begin{axis}[
        MyAxisStyle,
        title={$L^1$-TV}
        ]
        \pgfplotsset{every axis plot post/.append style={line width = 0.5pt}} ;
        \addplot+[only marks, mark size=0.5] table[x index=0,y=TVL1] {\figfolder/\noiseType2/\noiseType\noiseLevel.table};   
        \addplot+[no marks, dashed] table[x index=0,y index=1] {\figfolder/\noiseType2/\noiseType\noiseLevel.table};   
     \end{axis}
    \end{tikzpicture}
    &
     	\begin{tikzpicture}
      \begin{axis}[
        MyAxisStyleB,
        title={$L^2$-Potts}
        ]
        \pgfplotsset{every axis plot post/.append style={line width = 0.5pt}} ;
        \addplot+[only marks, mark size=0.5] table[x index=0,y=pottsL2] {\figfolder/\noiseType2/\noiseType\noiseLevel.table};   
        \addplot+[no marks, dashed] table[x index=0,y index=1] {\figfolder/\noiseType2/\noiseType\noiseLevel.table};   
     \end{axis}
    \end{tikzpicture} 
   &
     	\begin{tikzpicture}
      \begin{axis}[
        MyAxisStyleB,
 title = $L^1$-Potts
        ]
        \pgfplotsset{every axis plot post/.append style={line width = 0.5pt}} ;
        \addplot+[only marks, mark size=0.5] table[x index=0,y=pottsL1] {\figfolder/\noiseType2/\noiseType\noiseLevel.table};   
        \addplot+[no marks, dashed] table[x index=0,y index=1] {\figfolder/\noiseType2/\noiseType\noiseLevel.table};   
     \end{axis}
    \end{tikzpicture}
   \end{tabular}
   \figspace
    \caption{Laplacian noise with $\sigma = \noiseLevel$ (ground truth: dashed line.)
    The $L^1$-TV reconstruction exhibits staircasing effects.
   The $L^2$-Potts reconstructions is not robust to strong outliers.    The $L^1$-Potts procedure recovers the original piecewise constant signal almost perfectly.
     }
     \figspace
     \label{fig:laplacian_noise}
\end{figure}

We compare the $L^1$-Potts functional, the $L^2$-Potts functional and the $L^1$-TV functional w.r.t. to
their capabilities to reconstruct jump-sparse signals 
under different types of noise.
{\noop
For $L^1$-TV minimization we used the implementation of \cite{clason2009duality} with a maximum of 100 outer and 1000 inner iterations.
In all experiments, 
we adjusted the model parameter $\gamma$ so that
the solutions were visually closest to the ground truth.}

Our first experiment deals with additive Laplacian noise; 
see \autoref{fig:laplacian_noise}.
Every data point is corrupted by a random variable which is distributed according to the Laplacian probability density of variance $\sigma^2$
\begin{equation*}
	p(x) = \tfrac{1}{\sqrt{2}\sigma} e^{-\frac{\sqrt{2}}{\sigma}|x|}, \quad\text{with }\sigma > 0.
\end{equation*}
Our next experiment considers salt and pepper type noise; see \autoref{fig:potts_sap_small}.
Here, a fraction $p$ of a signal $g$ is set to a random variable $\eta,$ 
i.e., the data are given by
\begin{align*}
	f_i = g_i, &\text{ with probability }1-p   \qquad\text{and}\qquad f_i = \eta \text{ with probability }p.
\end{align*}
In our experiments, $\eta$ is distributed uniformly in the interval $[0,1].$
\begin{figure}
\centering
\def\noiseLevel{0.4}
\def\thisdatafile{ExperimentsRev/SaPNoise/SaPNoise\noiseLevel.table}
\begin{tabular}{cccc}
   	\begin{tikzpicture}
      \begin{axis}[
        MyAxisStyle,
        title=Data
        ]
        \pgfplotsset{every axis plot post/.append style={line width = 0.5pt}} ; 
        \addplot+[only marks, mark size=0.5] table[x index=0,y index=2] {\thisdatafile};   
     \end{axis}
    \end{tikzpicture}
   &
   \begin{tikzpicture}
      \begin{axis}[
        MyAxisStyle,
        title=$L^1$-TV
                ]
        \pgfplotsset{every axis plot post/.append style={line width = 0.5pt}} ;
        \addplot+[only marks, mark size=0.5] table[x index=0,y=TVL1] {\thisdatafile};   
        \addplot+[no marks, dashed] table[x index=0,y index=1] {\thisdatafile};   
     \end{axis}
    \end{tikzpicture}
    &
     	\begin{tikzpicture}
      \begin{axis}[
        MyAxisStyleB,
        title=$L^2$-Potts
        ]
        \pgfplotsset{every axis plot post/.append style={line width = 0.5pt}} ;
        \addplot+[only marks, mark size=0.5] table[x index=0,y=pottsL2] {\thisdatafile};   
        \addplot+[no marks, dashed] table[x index=0,y index=1] {\thisdatafile};   
     \end{axis}
    \end{tikzpicture}
   &
     	\begin{tikzpicture}
      \begin{axis}[
        MyAxisStyleB,
        title=$L^1$-Potts
                ]
        \pgfplotsset{every axis plot post/.append style={line width = 0.5pt}} ;
        \addplot+[only marks, mark size=0.5] table[x index=0,y=pottsL1] {\thisdatafile};   
        \addplot+[no marks, dashed] table[x index=0,y index=1] {\thisdatafile};   
     \end{axis}
    \end{tikzpicture}
   \end{tabular}
   \figspace
    \caption{
A signal corrupted by salt and pepper noise: $40 \%$ of the data points are destroyed. 
The $L^2$-Potts solution has few of the jumps of the original signal.
The $L^1$-TV reconstruction gets some plateaus but several outliers remain. 
The $L^1$-Potts solution yields an almost perfect reconstruction.
}
 \label{fig:potts_sap_small}
 \figspace
\end{figure}
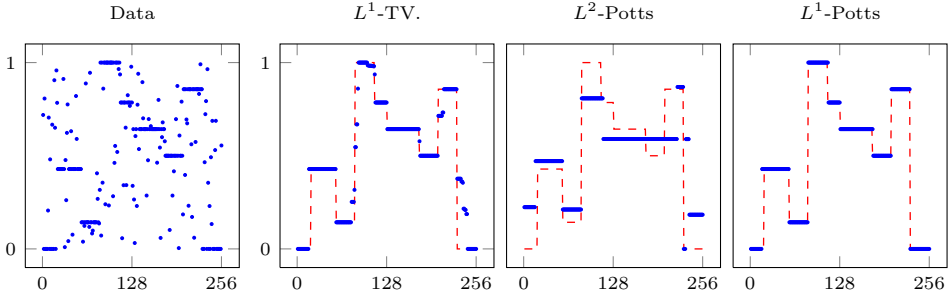
Our last experiment compares the reconstruction methods under Gaussian noise of variance $\sigma^2,$
see \autoref{fig:potts_gaussian_small}.
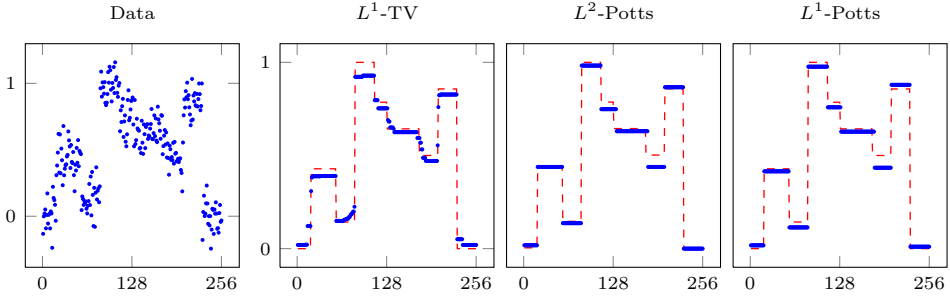
\begin{figure}
\centering
\def\noiseLevel{0.1}
\def\thisdatafile{ExperimentsRev/WhiteNoise/WhiteNoise\noiseLevel.table}
\begin{tabular}{cccc}
   	\begin{tikzpicture}
      \begin{axis}[
      MyAxisStyle,
        title=Data
                ]
        \pgfplotsset{every axis plot post/.append style={line width = 0.5pt}} ;
        \addplot+[only marks, mark size=0.5] table[x index=0,y index=2] {\thisdatafile};   
     \end{axis}
    \end{tikzpicture}
   &
     \begin{tikzpicture}
      \begin{axis}[
MyAxisStyle,
        title= $L^1$-TV
        ]
        \pgfplotsset{every axis plot post/.append style={line width = 0.5pt}} ;
        \addplot+[only marks, mark size=0.5] table[x index=0,y=TVL1] {\thisdatafile};   
        \addplot+[no marks, dashed] table[x index=0,y index=1] {\thisdatafile};   
     \end{axis}
    \end{tikzpicture}
    &
     	\begin{tikzpicture}
      \begin{axis}[
MyAxisStyleB,
title = $L^2$-Potts
                ]
        \pgfplotsset{every axis plot post/.append style={line width = 0.5pt}} ;
        \addplot+[only marks, mark size=0.5] table[x index=0,y=pottsL2] {\thisdatafile};   
        \addplot+[no marks, dashed] table[x index=0,y index=1] {\thisdatafile};   
     \end{axis}
    \end{tikzpicture}
  &
     	\begin{tikzpicture}
      \begin{axis}[
MyAxisStyleB,
        title=$L^1$-Potts
        ]
        \pgfplotsset{every axis plot post/.append style={line width = 0.5pt}} ;
        \addplot+[only marks, mark size=0.5] table[x index=0,y=pottsL1] {\thisdatafile};   
        \addplot+[no marks, dashed] table[x index=0,y index=1] {\thisdatafile};   
     \end{axis}
    \end{tikzpicture}
   \end{tabular}
   \figspace
    \caption{
    In presence of Gaussian noise with $\sigma = \noiseLevel,$
    both the $L^1$- and the $L^2$-Potts algorithm give almost perfect reconstructions.
    The $L^1$-TV reconstruction does not find all plateaus.
 } \label{fig:potts_gaussian_small}
 \figspace
\end{figure}
Summing up, we observe that, in all three experiments, the minimizers of the $L^1$-Potts functional
approximate the ground truth very closely. For Gaussian noise the $L^2$ and the $L^1$-Potts perform equally
well, whereas, for Laplacian and salt and pepper noise, $L^1$ Potts minimization turns out to be the method of choice.

\section{Robustness of the $L^1$-Potts functional to mildly blurred data}\label{sec:smallSizeKernels}

In this section, we consider the $L^1$-Potts functional \eqref{eq:potts_continuous_L1} 
associated with blurred data $f= K*g$. Here $g$ denotes a piecewise constant function on $[0,1]$ and 
$K$ is a convolution kernel. Concerning $K$ we assume the following throughout this section. 

\begin{assumption}\label{as:KernelK}
{\normalfont
We assume that $K$ is symmetric, compactly supported, 
nonnegative with $K>0$ in a neighborhood of the origin, and that it has total mass $1.$ }
\end{assumption}

{\noop In order to properly define the convolution product $K \ast g$ 
we extend the piecewise constant function $g$ to the left of $0$ and to the right of $1$ 
by its value on the leftmost and rightmost interval, respectively. 
(With minor technical modifications our results are also valid for periodical extensions, zero padding and for extensions by reflection.)   }

From \autoref{fig:potts_deconv_small} we see that neither $L^1$-TV nor $L^2$-Potts minimization are able
to reconstruct such blurred data $f= K*g.$ For $L^1$-Potts minimization, however, \autoref{fig:potts_deconv_small}
suggests that  it has some built-in blind deconvolution property: Without knowing the kernel $K,$ the signal $g$
is reconstructed  from data $f= K*g.$ We can even prove this blind deconvolution property:
There are parameters $\gamma > 0$ and a maximal kernel size $\kappa >0$
such that for all convolution kernels $K$ with smaller support size 
	\begin{align}\label{eq:potts_deconv_intrinsic}
		g = \argmin_{u\in \PC}   \gamma J(u)  + \| u - f\|_1 , \quad \text{ where } f = K \ast g.
	\end{align}

This statement and the corresponding discretized version are the main results of this section;
the precise formulations are \autoref{thm:deconvolution_main_extended} and  
\autoref{thm:deconvolution_main_extended_discr}.
{\noop We note that for our results the specific type of the kernel is irrelevant; 
e.g., truncated Gaussian or indicator functions may be the underlying unknown kernel.}

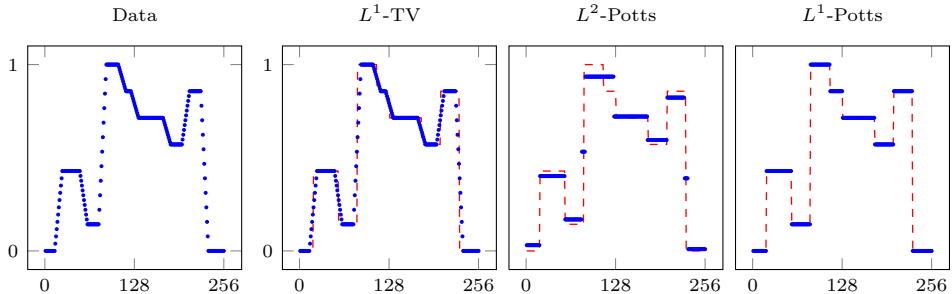
\begin{figure}
\centering
\def\convolutionSize{11}
\def\thisdatafile{ExperimentsRev/Convolution/pottsConvolution\convolutionSize.table}
\begin{tabular}{cccc}
   	\begin{tikzpicture}
      \begin{axis}[
        MyAxisStyle,
        title=Data
        ]
        \pgfplotsset{every axis plot post/.append style={line width = 0.5pt}} ;
        \addplot+[only marks, mark size=0.5] table[x index=0,y index=2] {\thisdatafile};   
     \end{axis}
    \end{tikzpicture}
    & 
    \begin{tikzpicture}
      \begin{axis}[
        MyAxisStyle,
        title=$L^1$-TV
        ]
        \pgfplotsset{every axis plot post/.append style={line width = 0.5pt}} ;
        \addplot+[only marks, mark size=0.5] table[x index=0,y=TVL1] {\thisdatafile};   
        \addplot+[no marks, dashed] table[x index=0,y index=1] {\thisdatafile};   
     \end{axis}
    \end{tikzpicture}
 & 
     	\begin{tikzpicture}
      \begin{axis}[
        MyAxisStyleB,
        title=$L^2$-Potts
        ]
        \pgfplotsset{every axis plot post/.append style={line width = 0.5pt}} ;
        \addplot+[only marks, mark size=0.5] table[x index=0,y=pottsL2] {\thisdatafile};   
        \addplot+[no marks, dashed] table[x index=0,y index=1] {\thisdatafile};   
     \end{axis}
    \end{tikzpicture}
   & 
     	\begin{tikzpicture}
      \begin{axis}[
        MyAxisStyleB,
       title=$L^1$-Potts
        ]
        \pgfplotsset{every axis plot post/.append style={line width = 0.5pt}} ;
        \addplot+[only marks, mark size=0.5] table[x index=0,y=pottsL1] {\thisdatafile};   
        \addplot+[no marks, dashed] table[x index=0,y index=1] {\thisdatafile};   
     \end{axis}
    \end{tikzpicture}
   
   \end{tabular}
    \caption{
    Blurred data $f = K*g,$ where $K$ is the moving average of size $\convolutionSize.$ 
    The $L^1$-TV method reconstructs in the blurred signal itself.
    The $L^2$-Potts method does not only not reconstruct all original plateaus,
    but even adds non-existing ones.
    The minimizer of the $L^1$-Potts recovers the piecewise constant signal $g$ perfectly.
 } \label{fig:potts_deconv_small}
\end{figure}

Furthermore, the experiments in \autoref{fig:introductory_example} 
suggest, that the blind deconvolution property of the $L^1$-Potts solutions even persists under noise. 
Here the quality of the reconstructions barely depends on the type of noise.

In order to show \eqref{eq:potts_deconv_intrinsic},
we first need a series of lemmas.
{\noop In these lemmas we throughout consider non-constant piecewise constant functions $g$
since, for constant signals $g$, \eqref{eq:potts_deconv_intrinsic} is obviously fulfilled.}
We denote the jump locations of the piecewise constant function $g$ by $0 < j_1 < ... < j_n < 1$
and let $l_{\min}$ be the minimal interval length given by this partition; cf. \autoref{fig:notation}.

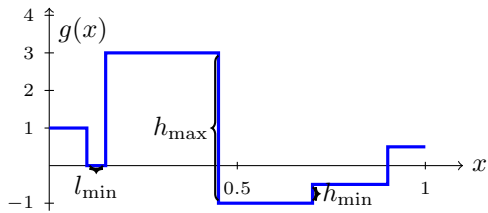
\begin{figure}[h!]
\begin{minipage}{0.66\textwidth}
\centering
	  \begin{tikzpicture}[x=5cm,y=0.5cm]
	\draw[->] (0,0) -- (1.1,0) node[right] {$x$};
	\draw[->] (0,-1.2) -- (0,4.2) node[below right ] {$g(x)$};
	\foreach \x in {0.5,1}
	\draw (\x,-.1) -- (\x,.1) node[below=4pt] {$\scriptstyle\x$};
	\foreach \y in {-1,1,2,3,4}
	\draw (-.0125,\y) -- (.0125,\y) node[left=4pt] {$\scriptstyle\y$};
	
	\draw [draw=blue,very thick] 
	(0, 1) -- ++ (0.1, 0) -- ++ (0, -1) node[inner sep=0pt] (l1) {} -- ++ (0.05,0) node[inner sep=0pt] (l2) {} -- ++ (0, 3) -- ++ (0.3, 0) node [inner sep=0pt] (H1) {} -- ++  (0, -4) node (H2) [inner sep=0pt] {} --++ (0.25,0) node [inner sep=0pt] (h1) {} -- ++ (0,0.5) node [inner sep=0pt] (h2) {}  --++ (0.2,0) --++ (0,1) --++ (0.1,0);
	\draw [thick,decorate,decoration={brace,mirror}] (l1) -- (l2) node [below,black,midway] {$l_{\min}$};
	\draw [thick,decorate,decoration={brace,mirror}] (H1) -- (H2) node [left,black,midway] {$h_{\max}$};
	\draw [thick,decorate,decoration={brace,mirror}] (h1) -- (h2) node [right,black,midway] {$h_{\min}$};
		 \end{tikzpicture}
	 \end{minipage}
		 \begin{minipage}{0.33\textwidth}
	 \caption{Notation used in \autoref{sec:smallSizeKernels}. $l_{\min}$ denotes the minimum interval length, 
	 $h_{\max}$ the maximum jump height and $h_{\min}$ the minimum jump height of a piecewise constant function $g.$ }
	 \label{fig:notation}	 
\end{minipage}
\end{figure}
For positive $\kappa \leq \frac{l_{\min}}{6}$ we consider
intervals $J_i$ around the jump location $j_i$ given by
\begin{align} \label{eq:defJi}
	J_i = [j_i - \kappa, j_{i} + \kappa].
\end{align}
We  denote the complementary intervals $N_i$ lying in between the jump locations by
\begin{align} \label{eq:defNi}
	N_i = (j_{i-1} + \kappa, j_{i} - \kappa).
\end{align}
On the left boundary, we set $N_1 = [0, j_1)$ and on the right boundary $N_{n+1} = (j_n + \kappa, 1].$

\begin{lemma}\label{lem:shift_lemma1}
	For  any $u_0 \in \PC$, there is $u_1 \in \PC$ 
	such that $u_1$ has no jump in the union $\bigcup_i N_i$ and 
	such that 
$
		P_\gamma(u_1) \leq P_\gamma(u_0).
$
\end{lemma}
\begin{proof}
    Consider an arbitrary index $i$ and
	assume that $u_0$ jumps in $N_i.$
	We first consider the case where $u_0$ has more than one jump in $N_i.$
	In this case, we define $u_1 = f$ on $N_i$  and $u_1 = u_0$ elsewhere. 
	Since $f=g$ is constant on $N_i,$ we have  $J(u_1) \leq J(u_0)$ 
	and,   $\norm{f-u_1}_1 \leq \norm{f-u_0}_1.$ 
	Hence $P_\gamma(u_1) \leq P_\gamma(u_0)$ which shows the lemma in this case.

	It remains to prove the case where $u_0$ has exactly one jump in $N_i,$ say at $x \in N_i = (l,r).$
	To this end we compare the value $d_l = |f-u_0|=|g-u_0|$ on $(l, x)$ with the value $d_r =  |f-u_0| = |g- u_0|$ on $(x, r).$
	If $d_l \geq d_r$ we shift the jump position to the left, 
	that means, we let $u_1 = u_0(x^+)$ on $N_i$ and $u_1 = u_0$ else. 
	The number of jumps 
	does not increase,
	$|f-u_1| \leq |f-u_0|$ on $N_i$ and thus
	\[
		\norm{f-u_1}_1 = \|(f-u_0)\mid_{N_i^C}\|_1  +  \|(f-u_1)\mid_{N_i}\|_1  \leq \norm{f-u_0}_1.
	\]
	If $d_l < d_r,$ we define $u_1 = u_0(x^-)$ on $N_i$ and $u_1 = u_0$ elsewhere.
	Then, an analogous argument completes the proof.
\end{proof}

 We use the notation $h_{\max}$ and $h_{\min}$ for the maximal and minimal jump height of $g,$
 respectively.
\begin{lemma}\label{lem:at_least_one_jump}
	Let $u_1 \in \PC$ with no jumps in $\bigcup_i N_i.$ 
	If
	\begin{align}\label{eq:assumption_lemma_at_least_one_jump}
		6 h_{\max} \kappa \leq \tfrac12 h_{\min} (l_{\min} - 2\kappa) - \gamma,
	\end{align}
	then there is $u_2 \in \PC$ such that
	$
		P_\gamma(u_2) \leq P_\gamma(u_1)
	$
	and such that $u_2$ has at least one jump in each $J_i.$ 
\end{lemma}

{\em Proof.}
	Assume that $u_1$ has no jump in $J_{i_0}$ for some $i_0$. 
	{\noop Then we find integers $k \geq 0$ and $s \geq 1$ 
	such that $u_1$ is constant on $N_{i_0-k} \cup J_{i_0-k} \cup \ldots \cup J_{i_0 + s - 1} \cup N_{i_0+s},$
	such that $u_1$ jumps in $J_{i_0-k-1}$ or $N_{i_0-k}$ is the leftmost interval 
	and such that $u_1$ jumps in $J_{i_0+s}$ or $N_{i_0+s}$ is the rightmost interval.
	We denote by $x_l$ the rightmost jump point of $u_1$ in $J_{i_0-k-1}$ (or by $0,$ if $N_{i_0-k}$ is the leftmost interval)
	and by $x_r$ the leftmost jump point of $u_1$ in $J_{i_0+s}$ (or by $1,$ if $N_{i_0+s}$ is the rightmost interval). 
	We define $u_2$ by
	$u_2 = g$ on $(x_l, x_r)$ and else by $u_2=u_1.$ 
	Obviously, $u_2$ has at least one jump in each $J_{i},$ for $i = i_0-k, \ldots, i_0+s -1.$}
	Furthermore, we have
	\begin{align}\label{eq:j_estimate}
		J(u_2) \leq J(u_1) + k + s.
	\end{align}
	Now we estimate the deviation of $u_2$ to $f = K \ast g$ on $(x_l, x_r)$ to obtain
	\begin{align}\label{eq:u_2estimate}  \notag
		\| (u_2 - K \ast g)|_{(x_l, x_r)} \|_1 &= \| (g - K \ast g)|_{(x_l, x_r)} \|_1 \\ 
		&\leq (k+s+2) h_{\max} 2 \kappa \leq 6 (k+s) h_{\max} \kappa.
	\end{align}
	This is because $g$ equals $K \ast g$ on each $N_i$
	and $\|(g-K*g)|_{J_i}\|_1$ is bounded by the jump height of $g$ in $J_i$ times the length of $J_i$
	which in turn equals $2 \kappa$ by definition.
	
	In order to estimate the deviation of $u_1$ to $K \ast g$ on $(x_l, x_r)$ 
	we need the following preparation.
	Let us consider a data sequence $g_0,...,g_n$ and an arbitrary $y \in \R.$
	Using the triangle inequality, we have,  
	for two consecutive members $g_j$ and $g_{j+1},$ 
	that $|g_j - y| + |g_{j+1} - y| \geq \min_{i} |g_i - g_{i+1}|.$
	This implies
	\begin{align}\label{eq:obvious_median_ineq}
		\sum\nolimits_{i= 0}^n |g_i - y| \geq \tfrac{n}{2} \min\nolimits_{i} |g_i - g_{i+1}|.
	\end{align}
	Now we come back to the estimate of $\| (u_1 - K \ast g)|_{(x_l, x_r)} \|_1.$
	 Recalling that $u_1$ is constant on $(x_l, x_r)$ 
	 (with possibly $x_l=0$ and $x_r=1$ as explained above)
	 and that $f = K*g$ equals $g$ on each $N_i,$
	 we see that
	{\noop \begin{align}\label{eq:u_1estimate1}
		\| (u_1 - K \ast g)|_{(x_l, x_r)} \|_1 \geq \sum_{i = i_0-k}^{i_0 + s} \|(u_1 - K*g)|_{N_i}\|_1 
			\geq  \min_i{|N_i|} \sum_{i = i_0-k}^{i_0 + s} |y - g_i| 
    \end{align}
    where} $y$ is the function value of $u_1$ on $(x_l, x_r)$ 
    and $g_i$ is the function value of $g = K*g$ on $N_i,$ where $g$ is constant on.
    We employ \eqref{eq:obvious_median_ineq} to get
		\begin{align}\label{eq:u_1estimate2}
		 \sum\nolimits_{i = i_0-k}^{i_0 + s} |y - g_i|  \geq  \tfrac12 (k+s) \min_{i} |g_i - g_{i+1}| 
		\geq \tfrac12 (k+s) \, h_{\min}.
    \end{align}
    Plugging \eqref{eq:u_1estimate2} into \eqref{eq:u_1estimate1}
    {\noop and using $| N_i | \geq l_{\min} - 2\kappa$ we get}
    \begin{align}\label{eq:u_1estimate3}
		\| (u_1 - K \ast g)|_{(x_l, x_r)} \|_1 \geq   (l_{\min} - 2 \kappa)\tfrac{k+s}{2} h_{\min}.
    \end{align}
    
    Next we set up a lower bound for $\|u_2 - K * f\|_1$ in terms of $\|u_1 - K * f\|_1.$
    To this end,  we first restrict to the interval $(x_l, x_r)$ and obtain 
    \begin{align}\label{eq:u_2estimate_by_u_1}
    		\|(u_2 - K * g)|_{ (x_l, x_r)}\|_1 
		&\leq  (k+s) 6 h_{\max}  \kappa \notag \\
		&\leq (k+s)\p{ \tfrac12 {h_{\min} (l_{\min} - 2\kappa)} - \gamma} \notag \\
		&\leq \|(u_1 - K * g)|_{ (x_l, x_r)}\|_1 - (k+s)\gamma.
    \end{align}
    For the first inequality we used \eqref{eq:u_2estimate},
    for the second inequality \eqref{eq:assumption_lemma_at_least_one_jump},
    and for the last inequality \eqref{eq:u_1estimate3}.
    Since $u_2$ equals $u_1$ outside $(x_l, x_r)$, inequality \eqref{eq:u_2estimate_by_u_1} implies that
    \begin{align*}
    		\|u_2 - K * g\|_1 
		&= \|(u_1 - K * g)|_{[0,1] \setminus (x_l, x_r)}\|_1 +  \|(u_2 - K * g)|_{ (x_l, x_r)}\|_1 \\
		&\leq  \|(u_1 - K * g)|_{[0,1] \setminus (x_l, x_r)}\|_1 +  \|(u_1 - K * g)|_{ (x_l, x_r)}\|_1 - (k+s)\gamma \\
		&= \|(u_1 - K * g)\|_1 - (k+s)\gamma.
    \end{align*}
    Therefore, using \eqref{eq:j_estimate}, {\noop we have} 
    \begin{align*}
    	P_\gamma(u_2) &= J(u_2) +  \|(u_2 - K * g)\|_1  \\ 
    	&\leq J(u_1) + \gamma (k+s) +  \|(u_1 - K * g)\|_1 - \gamma(k+s) = P_\gamma(u_1). \quad
    \end{align*} 
    {\noop So far, we have shown how to modify $u_1$ such that it jumps in each of the 
    intervals $J_{i_0-k},\ldots, J_{i_0+s -1}$ without increasing $P_\gamma.$ The procedure may cancel existing jumps in $J_{i_0-k-1}$ and 
    $J_{i_0+s}$ (which only  happens if $u_1(x_l^-) = g(x_l^+)$ or $u_1(x_r^+) = g(x_r^-)$).  
    If there is only one jump in $J_{i_0-k-1}$ or $J_{i_0+s}$ and this jump is canceled, i.e., $u_2$ has no jump there,     
    we apply the procedure of this lemma a second time; this time to $u_2$ for the initial interval $J_{i_0-k-1}$ or $J_{i_0+s}.$
    Then the resulting $u'_2$ has no jump in $J_{i_0-k-1}$ or $J_{i_0+s}.$ 
    Cancellation in $J_{i_0-k}$ or $J_{i_0+s-1}$ cannot occur for $u'_2$ since $u_2(x_l^-) = g(x_l^-) \neq g(x_l^+)$
    or $u_2(x_r^+) = g(x_r^+) \neq g(x_r^-)$ this time. 
    Summing up, we found a way to introduce jumps on $J_{i_0-k},\ldots, J_{i_0+s -1}$ keeping at least one jump in the adjacent
    intervals $J_{i_0-k-1}$ and $J_{i_0+s}$ while not increasing $P_\gamma$.}
    \endproof
\begin{lemma}\label{lem:precisely_one_jump}
	Let  $u_2 \in \PC$ and assume \eqref{eq:assumption_lemma_at_least_one_jump} and {\noop $\kappa \leq \tfrac{1}{6} l_{\min}$.} 
	Then there is $u_3 \in \PC$
	such that $u_3$ has no jumps in $\bigcup_i N_i$
	and precisely one jump in each $J_i,$ and
	 \begin{equation*}
	 	P_\gamma(u_3) \leq P_\gamma(u_2), \qquad\text{if $2 \kappa h_{\max} \leq \gamma.$}
	 \end{equation*}
\end{lemma}
\begin{proof}
	By Lemma \ref{lem:shift_lemma1} we may assume that $u_2$ has no jump in $\bigcup_i N_i$
	and, by Lemma \ref{lem:at_least_one_jump}, that $u_2$ has at least one jump in each $J_i.$
	
	We want to find $u'$  such that $u'$ agrees with $g$ on $N_i$ for all $i$
	and such that $P_\gamma(u') \leq P_\gamma(u_2).$ 
	To this end, we denote by $l_i$ the rightmost jump point of $u_2$ in $J_{i-1}$ 
	and by $r_i$ the leftmost jump point of $u_2$ in $J_{i}.$
	We define $u'$ by $u'= g_i$ on $(l_i, r_i)$ and by $u'= u_2$ elsewhere.
	Here $g_i$ is the value of $g$ on $N_i.$ 
	The number of jumps of $u'$ 
	does not exceed that
	of $u_2.$
	We first observe that, since $|N_i| \geq l_{\min} - 2 \kappa$, 
	\[
		\| (g - u_2)|_{N_i} \|_1 \geq (l_{\min} - 2 \kappa) |g_i - (u_2)_i|,
	\]	
	where $(u_2)_i$ is the value of $u_2$ on $N_i.$
	Using this estimate and noting that $\| (g - u')|_{N_i} \|_1 = 0$ we obtain
	\begin{equation}\label{eq:u_prime_estimate}
	\begin{split}
		\| (g - u')|_{(l_i,r_i)} \|_1 
		=& \| (g - u')|_{(l_i,r_i) \setminus N_i} \|_1 \\
		\leq&  \| (g - u_2)|_{(l_i,r_i) \setminus N_i} \|_1 + \| (u_2 - u')|_{(l_i,r_i) \setminus N_i} \|_1  \\
		 &+ \| (g - u_2)|_{N_i} \|_1 - (l_{\min} - 2 \kappa) |g_i - (u_2)_i|.
	\end{split}
	\end{equation}
	Since $u' = g_i$ on $(l_i, r_i)$ and $|(l_i,r_i) \setminus N_i| \leq 4 \kappa$ we have
	\[
		  \| (u_2 - u')|_{(l_i,r_i) \setminus N_i} \|_1 \leq 4 \kappa |g_i - (u_2)_i|. 
	\]
	Plugging this into \eqref{eq:u_prime_estimate} we get
	\[
		\| (g - u')|_{(l_i,r_i)} \|_1 \leq \| (g - u_2)|_{(l_i,r_i)} \|_1 + (6 \kappa - l_{\min}) |g_i - (u_2)_i| \leq \| (g - u_2)|_{(l_i,r_i)} \|_1,
	\]
	{\noop since an assumption of the lemma was that $\kappa \leq \tfrac{l_{\min}}{6}.$}
	Therefore, $P_\gamma(u') \leq P_\gamma(u_2).$ 
	{\noop Proceeding likewise for all intervals $N_i$ (where we let $l_i=0$ and $r_i=1$ for the boundary intervals)
	we get $u'$ with the desired property.}
	
	Now we define the function $u_3$ and show that its Potts-value $P_\gamma(u_3) \leq P_\gamma(u').$
	To that end, we consider the $k$ intervals $J_i = [l_i, r_i]$ were the $u'$ has more than one jump and
	 define $u_3$ as follows: we set $u_3 =	g(l_i^-)$ on $(l_i, \frac{r_i + l_i} 2),$ 
		$u_3 = g(r_i^+)$ on $(\frac{r_i + l_i} 2, r_i)$ and $u_3 = u'$ elsewhere.
    Then, 
    $J(u_3) \leq J(u') - k.$
    Furthermore, on each such interval $J = [l, r],$ where $u'$ has more than one jump,
    we have that
    \begin{align}\label{eq:Newww}
    		\| (K * g - u_3)|_{J} \|_1 \leq \| (K * g - u')|_J\|_1 + 2 \kappa h_{\max}.
    \end{align}
  {\noop This is because $|K * g - u_3| \leq h_{\max}$ on $J$ 
  since both functions $K \ast g$ and $u_3$ take their values in the interval $[g(l_i^-),g(r_i^+)]$ whose length does not exceed $h_{\max}.$}  
 	Summing over all the $k$ intervals where $u'$ jumps more than once we get using \eqref{eq:Newww}
	\begin{align*}
		P_\gamma(u_3) 
		&= \gamma J(u_3) + \| (K * g - u_3) \|_1 \\
		&\leq \gamma (J(u') - k) + \| (K * g - u')|\|_1 + k \cdot 2 \kappa h_{\max} \\
		&= P_\gamma(u') + k (2 \kappa h_{\max} -\gamma) \leq P_\gamma(u'),
	\end{align*}
	since by hypothesis  $2 \kappa h_{\max} \leq \gamma.$
	Summing up, we have found a function $u_3$ with $P_\gamma(u_3) \leq P_\gamma(u') \leq P_\gamma(u_2),$
	which has exactly one jump in each interval $J_i.$
\end{proof}

Recall that the minimal and maximal jump height of $g$ 
are denoted by $h_{\min}$ and $h_{\max},$ respectively.
Furthermore, $l_{\min}$ is the minimal interval length between two jumps of $g$; see \autoref{fig:notation}.
\begin{theorem}\label{thm:deconvolution_main_extended}
	Let $g$ be a piecewise constant function on $[0,1]$ and $f = K * g$ 
	with $K$ as in Assumption \ref{as:KernelK}.
    If the support size $\kappa$ of the convolution  kernel  $K$ satisfies
    \begin{align}\label{eq:kappa_constraint}
    		\kappa \leq \frac{h_{\min} l_{\min}}{2 (8 h_{\max} + h_{\min})}
    \end{align}
    then $g$ is the unique minimizer of  the Potts functional $P_\gamma$ associated with $f,$ 
    i.e.,
      \begin{align}\label{eq:potts_functional_in_deconv_thm}
    		P_{\gamma}(g) = \inf_{u \in \PC} P_{\gamma}(u) = \inf_{u \in \PC} \gamma J(u) + \| u - f\|_1,
    \end{align}
    for any Potts parameter $\gamma$ satisfying
    \begin{align}\label{eq:gamma_constraint}
    		2\kappa h_{\max} \leq \gamma \leq \tfrac12 h_{\min} l_{\min} - (h_{\min} + 6 h_{\max}) \, \kappa. 
    \end{align}
\end{theorem}
\begin{proof}
  {\noop If $g$ is constant, $f=g$ for any kernel of any size and the assertion of the theorem 
  is obviously true. So we may assume that $g$ has at least one jump.} 
    
	We first notice that \eqref{eq:kappa_constraint} implies $\kappa \leq \tfrac{1}{6} l_{\min},$
	which we assumed at the beginning of this section.	By \eqref{eq:kappa_constraint} we get
	\begin{align*}
		2 \kappa h_{\max} \leq \tfrac12 h_{\min} l_{\min} - (h_{\min} + 6 h_{\max}) \, \kappa,
	\end{align*}
	which allows us to choose $\gamma$ according to \eqref{eq:gamma_constraint}.
	Manipulating the right hand inequality of \eqref{eq:gamma_constraint} yields that
	\begin{align*}
		 6 h_{\max} \kappa \leq \tfrac12 h_{\min} (l_{\min} - 2 \kappa) - \gamma.
	\end{align*}
	This is precisely the assumption on $\kappa$ and $\gamma$ in Lemma \ref{lem:at_least_one_jump}.
	Furthermore, the left hand inequality of \eqref{eq:gamma_constraint} ensures
	that Lemma \ref{lem:precisely_one_jump} is applicable.
	Now Lemma \ref{lem:shift_lemma1} allows us to search the minimizer of $P_\gamma,$ 
	cf. \eqref{eq:potts_functional_in_deconv_thm},
	in the class of functions of piecewise constant functions without
	jumps in $\bigcup_i N_i.$
	Then, by Lemma \ref{lem:at_least_one_jump} and Lemma \ref{lem:precisely_one_jump},
	we may restrict this class further to those functions having precisely one jump in each interval $J_i,$ i.e., 
	 \begin{align}\label{eq:constrained_set}
	 	\inf_{u \in \mathrm{PC}} P_\gamma(u) = \inf \{ P_\gamma(u) : u \text{ jumps in no $N_i$ and precisely once in each $J_i$}  \}.
	 \end{align} 
Now we consider such a function $u$ 
which is a member of the set defined by the right hand side of \eqref{eq:constrained_set} 
and assume that $u$ does not equal $g$ identically.
Let us denote the jump locations of $u$ by $x_1,...,x_n$
with $x_i \in J_i.$
Without loss of generality, we may assume 
that $u$ coincides with $g$ on each $N_i,$
since otherwise, redefining $u = g_i$ on ${(x_i, x_{i+1})}$ (and probably also on the boundary intervals) yields a smaller $P_\gamma$ value (cf. beginning of the proof of Lemma \ref{lem:precisely_one_jump}).  

Since $u$ does not equal $g,$ we find at least one interval $J_i$ where
the jump locations $j_i$ of $g$ and the corresponding jump location $x_i$ of $u$ do not coincide. We show that 
\begin{align}\label{eq:u_norm_larger_than_f_norm}
	\| (u - K * g)|_{J_i}\|_1 > \| (g - K * g)|_{J_i}\|_1
\end{align}
on each $J_i$ where the jumps of $g$ and $u$ do not coincide.
This implies, as $g$ and $u$ have the same number of jumps, that 
$P_\gamma(u) > P_\gamma(g).$

To complete the proof it remains to show  \eqref{eq:u_norm_larger_than_f_norm}.
To this end, we denote the left hand and right hand boundary of $J_i$ by $l$ and $r,$
respectively. We note that $g(l) = u(l^-),$ $g(r) = u(r^+),$
and $g$ jumps at $\tfrac{1}2 (l+r) = j_i.$ Recall that, by assumption, the jump location $x_i$ of $u$ fulfills $x_i \neq j_i.$
The symmetry of $K$ implies that the mass of $K$ on the negative half axis equals that on the positive half axis, which in turn yields
\begin{align}\label{eq:mitte_ist_mitte}
	K * g (\tfrac{1}{2}(l+r)) = \tfrac{1}{2} (g(r) + g(l)).
\end{align}
Since $K$ is positive, $K * g$ is monotone on $J_i.$ Together with the 
above equation \eqref{eq:mitte_ist_mitte}, this implies
\begin{align}\label{eq:J_ineq_not_strict}
	\| (g - K * g)|_{J}\|_1 \leq \tfrac12 |g(r) - g(l)| \cdot |J|, \qquad\text{ for any interval $J\subseteq J_i.$}
\end{align}
Furthermore, as we assume that $K > 0$ in a neighborhood of the origin
(cf. Assumption~\ref{as:KernelK}),     
$K*g$ is strictly monotone in a neighborhood of $j_i.$
This shows that inequality \eqref{eq:J_ineq_not_strict} is strict,
i.e., we may replace the symbol \enquote{$\leq$} in \eqref{eq:J_ineq_not_strict}  by \enquote{$<$}.
On the interval $I'$ which is given by the endpoints $x_i$ and $j_i$
it holds that $|u(x) - K*g(x)| > \frac12 |g(r) - g(l)|,$ for almost all $x.$
{\noop This follows from inspection of all possible cases: 
if $u(l)=g(l)< g(r)=u(r)$ and $x_i<j_i,$ 
then $K\ast g < \frac12 (g(r) + g(l)) $ on $I'$ and 
$u = g(r)$ on $I'$; hence, $u- K \ast g$ $> \frac12 (g(r) - g(l))$ on $I'.$ The other three cases are analogous. 
Summing up,} 
\begin{align}\label{eq:LastToInsert}
	\| (K * g - u)|_{I'} \|_1 > \tfrac{1}{2} |g(r) - g(l)| \cdot |I'| > \| (K * g - g)|_{I'} \|_1.
\end{align}
As assumed above, $u$ equals $g$ are on all $N_i,$ and both $u$ and $g$ jump precisely once in each $J_i.$ 
Then \eqref{eq:LastToInsert} implies that $u$ equals $g$ on $[0,1]$ for the following reason.
Otherwise, there were $J_i$ and a (proper) subinterval $I' \subset J_i$ (with positive Lebesgue measure)  
such that \eqref{eq:LastToInsert} holds true. This would imply \eqref{eq:u_norm_larger_than_f_norm} which, in turn, would contradict $u$ being a minimizer.
Hence, any minimizer equals $g$ which completes the proof.

\end{proof}

The next statement is the discrete version of \autoref{thm:deconvolution_main_extended}.
\begin{theorem}\label{thm:deconvolution_main_extended_discr}
	Let $g$ be a piecewise constant function on $[0,1]$ and let 
  $\kappa$ be the support size of the convolution kernel $K$ 
  fulfilling Assumption \ref{as:KernelK}.
  We recall that 
	$P^k_\gamma$ is the discrete Potts functional given by \eqref{eq:potts_fully_discrete} and is associated with the discretization
	set $X^k=\{x_1,\ldots,x_n\}$ of mesh size $\eta = \max_i |x_i-x_{i-1}|.$ We assume that    
    \begin{align}\label{eq:kappa_eta_constraint_disc}
    		\kappa+\eta \leq \frac{h_{\min} l_{\min}}{2 (8 h_{\max} + h_{\min})}.
    \end{align}
    Let $g^k= \sum_i \alpha_i \indfunc_{(s_i,s_{i+1})}$  be a \enquote{$X^k$-discretization}       
   of   $g = \sum_i \alpha_i \indfunc_{(t_i,t_{i+1})}$,
    where $s_i \in X^k$ is a nearest neighbor of the jump location $t_i$ of $g$.
    Furthermore, we consider data $f^k =$ $\sum_j S_k f(j) \cdot \indfunc_{I_j}$ obtained
    from sampling $f = K * g$ on the discretization set $X^k$ using 
    \eqref{eq:sampling_integral} or \eqref{eq:sampling_point} (with midpoint sampling).   
    Then each function $g^k$ is a minimizer of the discrete Potts functional $P^k_\gamma$ associated with data $f^k,$  i.e.,
      \begin{align}\label{eq:potts_functional_in_deconv_thm_disc}
    		P^k_{\gamma}(g^k) = \inf_{u \in \PCk} P^k_{\gamma}(u) = \inf_{u \in \PCk} \gamma J(u) + \| u - f^k\|_1,
    \end{align}
    for any Potts parameter $\gamma$ satisfying
    \begin{align}\label{eq:gamma_constraint_disc}
    		2(\kappa+\eta )h_{\max} \leq \gamma \leq \tfrac12 h_{\min} l_{\min} - (h_{\min} + 6 h_{\max}) \, (\kappa+\eta). 
    \end{align}
    The $X^k$-discretizations $g^k$ of the above type are the only minimizers, 
    hence the minimizer is unique whenever each jump location $t_i$ of $g$ has a unique nearest neighbor in $X^k.$   
    Furthermore, if $(X^k)$ is a nested sequence of discretization sets being eventually dense, each sequence of 
    minimizers $(g^k)$ converges to $g.$   
\end{theorem}
\begin{proof}
The proof of \autoref{thm:deconvolution_main_extended_discr} uses arguments analogous to those in the proof of 
\autoref{thm:deconvolution_main_extended}. In particular, the statements and proofs of 
Lemma \ref{lem:shift_lemma1}, Lemma \ref{lem:at_least_one_jump}, and Lemma \ref{lem:precisely_one_jump}
remain true in the setup
of \autoref{thm:deconvolution_main_extended_discr} if we replace each occurrence of $\kappa$ by $\kappa+\eta$
in these lemmas, their proofs and in the definition of $N_i$ and $J_i$ in  \eqref{eq:defJi} and \eqref{eq:defNi}, respectively.
For the rest of this proof, we define $N_i$ and $J_i$ accordingly with $\kappa+\eta$ replacing $\kappa.$ 
Therefore, it is sufficient to show
\eqref{eq:potts_functional_in_deconv_thm_disc} 
for the functions $u$ in $\PCk$
which jump in no $N_i$ and precisely once in each $J_i.$  Thus, let us consider a function $u$,
which we may assume, as in the proof of  \autoref{thm:deconvolution_main_extended},
to coincide with $g$ on each $N_i.$ 
{\noop We denote the jump location of $u$ in $J_i$ by $z_i$
 and the jump location of $g$ in $J_i$ by $t_i.$ 
 We use the notation $S_i$ for the set of nearest neighbors of $t_i$ in the $k$th level discretization set $X^k$ 
 (which lie in $J_i$ since we used $\kappa+\eta$ in the definition of $J_i$ at the beginning of this proof).
 We consider intervals $J_i,$ where $z_i \notin S_i,$ and show that
 \begin{align}\label{eq:u_norm_larger_than_f_norm_discA}
	\| (u - f^k)|_{J_i}\|_1 > \| (g^k - f^k)|_{J_i}\|_1.
 \end{align}
 Furthermore, we consider an additional $X^k$-discretization $\tilde{g}^k$ (with the same properties as $g^k$ but possibly different jump locations in $S_i$)
 and show that, on each $J_i,$  
 \begin{align}\label{eq:u_norm_larger_than_f_norm_discB}	
	\| (g^k - f^k)|_{J_i}\|_1 = \| (\tilde{g}^k - f^k)|_{J_i}\|_1.
\end{align} }
Since the number of jumps of $u,$ $g^k$ and $\tilde{g}^k$ are equal, 
\eqref{eq:u_norm_larger_than_f_norm_discA} and \eqref{eq:u_norm_larger_than_f_norm_discB} yield \eqref{eq:potts_functional_in_deconv_thm_disc}.

{\noop
In order to show \eqref{eq:u_norm_larger_than_f_norm_discA} and \eqref{eq:u_norm_larger_than_f_norm_discB}, we first note that, by the proof of
\autoref{thm:deconvolution_main_extended}, $f(t_i)$ $= K \ast g (t_i) $ $= \tfrac{1}{2} (g(t_i^-)+g(t_i^+))$ $=:a$ and
$f$ is monotone and strictly monotone in a neighborhood of $t_i.$ 
So, if $t_i$ has two nearest neighbors $l,r$ in $S_i,$ then the symmetry of $K$ implies that $f^k$ equals $a =\tfrac{1}{2} (g(t_i^-)+g(t_i^+))$
on the interval $(l,r).$ This is because $K \ast g (t_i-x) + K \ast g (t_i+x) = 2 a$ on $(l,r)$ and thus 
both the integral sampling \eqref{eq:sampling_integral} and the sampling \eqref{eq:sampling_point} using midpoints yield that $f^k$ equals $a$ on $(l,r)$. 
Furthermore, $f^k$ is monotone on $J_i$ since so is $f.$
(We recall that, in the beginning of the proof, we defined $J_i$  with $\kappa+\eta$ replacing $\kappa$ in \eqref{eq:defJi}.)
 Hence, $|g^k-f^k| \leq b $ $:= \tfrac{1}{2} |g(t_i^-)-g(t_i^+)|.$
Also, $f^k$ jumps at $l$ and $r$ since $f$ is strictly monotone in a neighborhood of $t_i.$
Hence, if the jump point $z_i$ of $u$ is not in  $S_i = \{l,r\},$ then  $|u - f^k| > b $ (a.e.) on the interval $B$ connecting $z_i$ 
and $n,$ 
where $n \in \argmin_{\nu \in \{l,r\}} |\nu - z_i|.$
This follows from an argument analogous to the differentiation of cases preceding \eqref{eq:LastToInsert}.
Since $u$ equals $g^k$ on $J_i \setminus B,$ we have shown \eqref{eq:u_norm_larger_than_f_norm_discA} when $t_i$ has two nearest neighbors.
Since $g^k$ and $\tilde{g}^k$ jump at $l$ or $r$ we have that  $|g^k - f^k|= b = |\tilde{g}^k - f^k|$ on $(l,r).$ Since $g^k$ and $\tilde{g}^k$ are equal on $J_i \setminus (l,r)$ we get
\eqref{eq:u_norm_larger_than_f_norm_discB}.

If $t_i$ has only one nearest neighbor $s_i$ in $X^k,$ 
then on the one side (left or right) of $s_i,$ we have $f^k>a,$ 
and on the other side (right or left) of $s_i$,  we have $f^k<a.$
To see this, denote the other neighbor (which is not a nearest neighbor) of $t_i$ in $X^k$ by $o_i.$
As seen in the proof of \autoref{thm:deconvolution_main_extended}, $f$ is monotone and strictly monotone in a neighborhood of $t_i.$
Then, the midpoint $m_i$ of $s_i$ and $o_i$ does not equal $t_i$ and hence $f^k(t_i) \neq a$ by the strict monotonicity of $f$ near $t_i,$
for both kind of sampling operators \eqref{eq:sampling_integral} and \eqref{eq:sampling_point} with midpoint sampling. 
Furthermore, $f^k(t_i) = f^k(m_i) > a,$ if $f(s_i)< a$ since $s_i$ is closer to $t_i$ than $o_i.$
Analogously, $f^k(t_i) < a,$ if $f(s_i)> a.$ 
(Note that $f(s_i)= a$ does not happen, since we consider the case where $s_i$ is unique.)
Then, if $f(s_i)< a,$ we have $f^k<a$ on that side of $s_i$ where $t_i$ does not lie,
and analogously (with inverse inequality) for the case $f(s_i)>a.$  
Now, we proceed as in the case of nonunique nearest neighbors above.
We assume that the jump point $z_i$ on $J_i$ of $u$ does not equal $s_i.$
Then we denote the interval connecting  $z_i$ and $s_i$ by $B.$  
On $B$, we have  $|u - f^k| > b $ (a.e.) on the interval $B$ connecting $z_i$ and $s_i.$
Since $g^k = u$ on $J_i \setminus B$ and  since  $|g^k - f^k| \leq b$ on $J_i$,
we have shown \eqref{eq:u_norm_larger_than_f_norm_discA} in the case where $t_i$ has only one nearest neighbor.}
The last assertion of the theorem is a consequence of the fact that the mesh sizes of $(X^k)$ tend to $0$ as $k \to  \infty$.
\end{proof}

We note that, although the statement of the theorem holds for both kinds of sampling operators introduced 
in \eqref{eq:sampling_integral} and \eqref{eq:sampling_point} with midpoint sampling, it is more natural
to consider point sampling here since averaging might be already modeled by the kernel $K.$

\section{Reconstruction of highly blurred jump-sparse signals}\label{sec:largeSizeKernels}

\begin{figure}
\centering
\def\noiseLevel{0.1}
\def\thisfigheight{0.3\textwidth}
\def\thisfigwidth{0.43\textwidth}
\def\thisdatafile{ExperimentsRev/iterativeDeconv2/LaplaceNoise\noiseLevel.table}
\begin{tabular}{ccc}
   	\begin{tikzpicture}
      \begin{axis}[
        MyAxisStyle,
        title=Data
        ]
        \pgfplotsset{every axis plot post/.append style={line width = 0.5pt}} ;
        \addplot+[only marks, mark size=0.5] table[x= x,y= fNoisy] {\thisdatafile};   
     \end{axis}
    \end{tikzpicture}
        &  
     	\begin{tikzpicture}
      \begin{axis}[
        MyAxisStyle,
        title=K-$L^1$-TV
        ]
        \pgfplotsset{every axis plot post/.append style={line width = 0.5pt}} ;
        \addplot+[only marks, mark size=0.5] table[x = x,y=TVL1K] {\thisdatafile};   
        \addplot+[no marks, dashed] table[x = x,y=original] {\thisdatafile};   
     \end{axis}
    \end{tikzpicture}
   &
     	\begin{tikzpicture}
      \begin{axis}[
                MyAxisStyleB,
                title= K-$L^1$-Potts
               ]
        \pgfplotsset{every axis plot post/.append style={line width = 0.5pt}} ;
        \addplot+[only marks, mark size=0.5] table[x =x,y=pottsL1K] {\thisdatafile};   
        \addplot+[no marks, dashed] table[x = x,y=original] {\thisdatafile};   
     \end{axis}
    \end{tikzpicture}
    \end{tabular}
    \figspace
    \caption{Signal blurred by a (normalized) Gaussian kernel
    and corrupted by Laplacian noise with $\sigma = \noiseLevel.$
    {\noop The minimizer of the $L^1$-TV deconvolution functional has additional jumps
    and less contrast at the large jumps.}
    The proposed $L^1$-Potts deconvolution strategy reconstructs the original signal almost perfectly.
     } \label{fig:potts_deconv_large}
     \figspace
\end{figure}
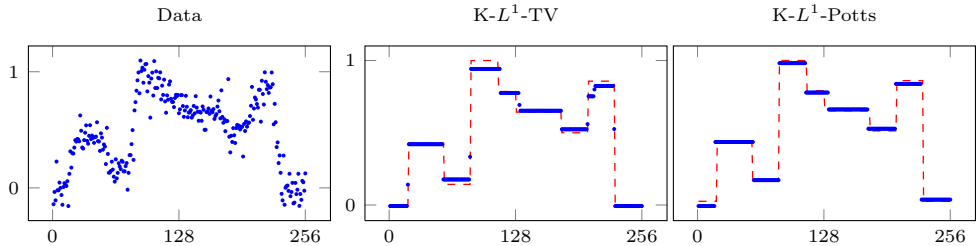

In the previous section we have seen that 
$L^1$-Potts minimization is robust to mildly blurred data.
In order to deal with severely blurred data,
we propose a heuristic splitting approach to the deconvolution problem \eqref{eq:potts_L1Conv}.
In contrast to the previous section, we need to know the kernel $K.$ 
We rewrite \eqref{eq:potts_L1Conv} as the constrained optimization problem in two variables 
\begin{align}
	\min_{u,v} \, \gamma J(u) + \| K * v - f \|_1,\quad \text{subject to}\quad u - v = 0.
\end{align} 
Then we add the constraint $u-v$ as the extra
penalty $\mu \| u-v\|_1$ to the target functional to obtain the unconstrained optimization problem 
\begin{align}\label{eq:split_functional}
	\gamma J(u) +  \mu \, \| u - v \|_1 +  \| K *  v - f\|_{1} \to \min.
\end{align}
The extra parameter $\mu >0$ controls how strongly $u$ and $v$  are tied to each other.
We approach the bivariate problem \eqref{eq:split_functional} by the following heuristic.
We first fix $v$ and minimize with respect to the variable $u;$ then we fix $u,$ and optimize 
w.r.t.\ $v.$ This yields the two univariate problems
\begin{align}\label{eq:split_functional_potts}
	\min_u \, \tfrac\gamma\mu J(u) +   \| u - v \|_1 \quad  \text{and} \quad
	\min_w \, \mu \| w \|_1 +  \| K * w - (K * u -  f)\|_{1},
\end{align}
where $w=u-v.$
The first problem is a $L^1$-Potts minimization problem for which we apply the fast algorithm developed in this paper.
The second problem is a convex optimization problem which we solve by using the corresponding algorithm in 
\cite{baritaux2011primal}; alternatives are  \cite{fu2006efficient, chambolle2011first}.
Summing up, we 
alternately minimize the two problems in \eqref{eq:split_functional_potts}
where we successively increase the coupling parameter $\mu$ (by a factor of 1.5) during the iteration.
{\noop In contrast to the rest of the paper, no convergence analysis for this method is carried out here.}

Numerical experiments comparing the results of the proposed splitting algorithm for the $L^1$-Potts deconvolution problem \eqref{eq:potts_L1Conv}
with the solutions of the $L^1$-TV problem
\begin{align}\label{eq:L1-TV-K}
	\gamma \| \nabla u \|_1 + \| K * u - f \|_1 \to {\min}
\end{align} 
can be found in \autoref{fig:potts_deconv_large}. To solve \eqref{eq:L1-TV-K} we used the algorithm proposed in \cite{clason2009duality}.  
We observe that the computed minimizer of the $L^1$-TV functional still contains noise, whereas the solution of our approach recovers the piecewise constant signal almost perfectly.

\section{Conclusion and future research}

We have seen that the $L^1$-Potts functional is particularly well suited to reconstruct piecewise constant functions under 
non-Gaussian noise.
 We  established a continuous $L^1$-Potts model with the feature that the limits of minimizers of its (possibly non-equidistant) discretizations
are minimizers of the continuous model. 
 For the solution of these discrete $L^1$-Potts functionals,
we have introduced an exact and fast algorithm.

We have shown that the $L^1$-Potts functional is robust to mild blurring.
For strongly blurred signals and known source of blurring,
we have proposed an algorithm to solve the associated $L^1$-Potts deconvolution problem. One direction of future research is to obtain analytic results for this heuristic approach.

Another direction of future research is the application of our computational strategies to higher dimensional problems. 
Minimizing the direct analogue of the Potts functional in higher dimensions is a
NP-hard problem \cite{boykov2001fast}.
However, the computational load may be reduced by restricting the set of admissible partitions.
One approach in that direction are wedgelets \cite{donoho1999wedgelets}. 
In the $L^2$ setting, a fast wedgelet algorithm has been obtained in \cite{friedrich2008efficient}.
Building on the ideas for the fast computations of medians in this paper,
the authors were able to derive a fast algorithm for the computation of $L^1$-wedgelets.
Its computational complexity lies in the order of $O(n \log n),$ 
where $n$ is the number of pixels. This result is the matter of a forthcoming paper. 
Here, a related open problem is to find fast algorithms for wider classes of admissible partitions.

\section*{Acknowledgement}
This work was supported by the German Federal Ministry for Education and Research under SysTec Grant 0315508. 
The first author acknowledges support by the Helmholtz Association within the young investigator group VH-NG-526. 
The second author acknowledges support from the European Research Council under the European Union's Seventh Framework Programme (FP7/2007-2013) / ERC grant agreement no.~267439. 
The authors would like to thank the anonymous reviewers
for their valuable comments and suggestions which helped to improve the presentation of the paper.


{
\bibliographystyle{siam}
\bibliography{fastPotts}
}

\end{document}